\documentclass[11pt, dvipsnames]{amsart}
\usepackage[utf8]{inputenc}
\usepackage{subcaption}
\usepackage{graphicx}

\usepackage{graphicx}
\usepackage{grffile}
\usepackage{graphbox}

\usepackage{a4wide}
\usepackage{color}
\usepackage{xcolor}
\usepackage{amsmath,amssymb}
\usepackage{float}
\usepackage{soul}
\usepackage{lipsum}
\usepackage{hyperref}
\usepackage{cleveref}

\usepackage[foot]{amsaddr}

\usepackage{pgfplots}
\usepgfplotslibrary{fillbetween}
\pgfplotsset{compat=newest}
\usepackage{pgfplotstable}
\usepgfplotslibrary{groupplots}
\usepackage{diagbox}

\usepackage{booktabs}  
\usepackage{capt-of}
\usepackage{multirow}
\usepackage[mathlines]{lineno}

\usepackage{float}
\usepackage{commath}

\usepackage{amsthm}
\newtheorem{theorem}{Theorem}[section]
\newtheorem{lemma}[theorem]{Lemma}

\newtheorem{assumption}{Assumption}

\newtheorem{remark}{Remark}

\numberwithin{equation}{section} 

\crefname{section}{Section}{Sections}
\crefname{subsection}{Section}{Sections}
\crefname{subsubsection}{Section}{Sections}

\Crefname{section}{Section}{Sections}
\Crefname{subsection}{Section}{Sections}
\Crefname{subsubsection}{Section}{Sections}

\crefname{example}{Example}{Examples}

\newcommand{\jump}[1]{\ensuremath{[#1]} }
\newcommand{\avg}[1]{\ensuremath{\left\{#1\right\}}}

\newcommand{\bm}[1]{\boldsymbol{#1}} 
\newcommand{\enorm}[1]{{\left\vert\kern-0.25ex\left\vert\kern-0.25ex\left\vert #1 
    \right\vert\kern-0.25ex\right\vert\kern-0.25ex\right\vert}}
 \newcommand{\vertiii}[1]{{\left\vert\kern-0.25ex\left\vert\kern-0.25ex\left\vert #1 
    \right\vert\kern-0.25ex\right\vert\kern-0.25ex\right\vert}}%

\newcommand{\DG}{\mathrm{DG}}

\newcommand{\bfsigma}{\boldsymbol{\sigma}}

\newcommand{\bfn}{\bm n}

\newcommand{\bfw}{\bm w}

\newcommand{\Bk}{\color{black}}

\usetikzlibrary{shapes, patterns, arrows}
\usetikzlibrary{calc}

\title[DG on networks]{Analysis of a Discontinuous Galerkin Method for Diffusion Problems on Intersecting Domains}

\author{Miroslav Kuchta$^1$, Rami Masri$^2$, and  Beatrice Riviere$^3$}

\email{miroslav@simula.no, rami\_masri@brown.edu, riviere@rice.edu}
\address{$^1$Department of Numerical Analysis and Scientific Computing, Simula Research Laboratory, Oslo, 0164 Norway. MK acknowledges support from the Research Council of Norway grant 303326.}
\address{$^2$Division of Applied Mathematics, Brown University, Providence, RI 02906, USA.
  }
\address{$^3$Department of Computational Applied Mathematics and Operations Research, Rice University, Houston, TX 77005, USA. The work of BR is partially supported by the National Science Foundation grant DMS-2513092.
}
\date{ \today }

\begin{document}

\maketitle

\begin{abstract}
The interior penalty discontinuous Galerkin method is applied to solve elliptic equations on either networks of segments or networks of planar surfaces, with arbitrary but fixed number of bifurcations. Stability is obtained by proving a discrete Poincar\'e's inequality on the hypergraphs.  Convergence of the scheme is proved 
for $H^r$ regularity solution with $1 < r \leq 2$. In the low regularity case  ($r \leq 3/2$), a weak consistency result is obtained via generalized lifting operators for Sobolev spaces defined on hypergraphs. Numerical experiments confirm the theoretical results.
\smallskip
\noindent {\sc{MSCcodes.}}   65M60, 65N30, 53Z99, 57N99
\end{abstract}

\newtheorem{thm}{Theorem}[section]

\section{Introduction}\label{sec:intro}

There is a need for the numerical analysis of elliptic partial differential equations (PDEs) in non-standard computational domains, namely, networks of manifolds of lower dimensions.  This type of problem arises in many application domains, for instance, from geosciences, where the network of planar surfaces may represent  natural fractures in the subsurface (see \cite{berre2019flow} and references therein).  Another application is in
biomedicine, where the network of one-dimensional lines is a representation of the vasculature of an organ: indeed, each blood vessel is reduced to its centerline \cite{laurinozunino2019,amare20251d,masri2024modelling,causemann2025silico}. One main advantage of using topological model reduction to obtain PDEs on manifolds of lower dimension is the gain in computational efficiency.  

In this work, we formulate and analyze a discontinuous Galerkin method for solving the diffusion problem in two types of networks: (i) a network made of connected one-dimensional lines and (ii) a network made of connected two-dimensional planar polygonal surfaces. In either case, the computational domain is a hypergraph with an arbitrary but fixed number of graph edges and graph nodes including bifurcation nodes.
Standard interior penalty discontinuous Galerkin  (DG) methods assume that an
  interior facet is shared by exactly two elements. However, hypergraphs lead to meshes where this assumption is no longer valid.
  To discretely enforce the conditions at the bifurcation graph nodes, a special flux is constructed that generalizes the discontinuous Galerkin fluxes usually used on standard computational domains. With a generalized definition of jumps and averages at the bifurcation nodes, the interface conditions are seamlessly written in the DG form. We introduce an enriching map that is the key ingredient in the proof of the discrete Poincar\'e's inequality on hypergraphs, used to obtain the stability of the numerical solutions. 
  
  Convergence of the scheme is established for the case of a smooth solution and for the case of a low regularity solution. The numerical analysis for the low regularity solution is non-standard as Galerkin orthogonality is no longer valid. The proof utilizes and extends to hypergraphs the functional analysis tools introduced in \cite{ern2021quasi} for classical domains. They include the generalization of lifting operators and mollification operators to networks of planar surfaces.  By defining a weak notion of the normal traces, we extend the DG bilinear form so that the first argument belongs to a space on the hypergraph of low regularity functions. We then obtain a weak consistency result and we derive a priori error bounds.  

As mentioned above, this work is relevant to the study of flows in fractured porous media for the geosciences applications. In that case, flow inside the network of fractures is coupled with the flow in the surrounding rock matrix. The literature on mixed-dimensional models of Darcy flow in fractured porous media is vast, specially in the simpler case of one single fracture.  Over the last decade, there has been an increase on the numerical investigation of the more challenging case of networks of fractures embedded into a porous medium, see for instance the
works \cite{sandve2012efficient, antonietti2016mimetic, brenner2017gradient,boon2018robust,nordbotten2019unified,antonietti2019discontinuous,antonietti2022polytopic} and the references therein. 
Our work can also be placed in the broader context of numerical analysis of surface PDEs. The literature on the Laplace-Beltrami PDE is significant: for instance, finite element method \cite{Dziuk2006,guzman2018analysis,bonito2020finite}, discontinuous Galerkin methods \cite{dedner2013analysis,CockburnDemlow2016,burman2017cut} and other discretization methods have been studied.  On one hand, because the domains are one dimensional lines or two dimensional planar polygons, the use of the Laplace-Beltrami operator is not needed here and this facilitates the numerical analysis of the PDEs. On the other hand, the fact that the lines or surfaces intersect and form a hypergraph introduces challenges in the numerical analysis. Additional constraints on the bifurcation nodes are required and special treatment of the interface conditions is needed to ensure well-posedness and convergence of the numerical methods. 

The literature on theoretical numerical analysis of PDEs on hypergraphs remains scarce. In \cite{rupp2022partial,Knobloch2025}, the authors analyze a mixed hybridizable discontinuous Galerkin method for elliptic equations on hypergraphs (network of lines or network of planar surfaces). In \cite{rupp2022partial},  the Kirchoff interface condition at the bifurcation nodes of the hypergraph is relaxed by the addition of  a concentrated source.  Another DG method is formulated in \cite{masri2024discontinuous} for networks of 1D lines: in that work, hybridization is occurring only at the bifurcation nodes and only existence and uniqueness of the discrete solution is shown. The work \cite{hansbo2017nitsche} presents a
continuous finite element method to solve the diffusion problem on composite surfaces. The authors propose a Nitsche formulation of the Kirchoff law at the interfaces of the  composite surfaces that can be viewed (after some algebraic manipulation) as a generalization of our approach for handling the Kirchoff condition. With Galerkin orthogonality, 
error bounds are derived for smooth enough solutions, namely functions in $H^{1+s}$ on each surface for $1/2 < s \leq 1$. Conservation laws on networks of 1D lines are numerically discretized by a finite volume method in \cite{fjordholm2022well}; convergence of the numerical solution to a unique entropy solution is proved. PolyDG methods are analyzed for coupled flows in \cite{antonietti2022polytopic}
under the assumption of smooth enough solutions, and for the simpler case of a hypergraph with only one bifurcation node.


The outline of the paper is as follows. \Cref{sec:continuous} defines the PDEs, the interface conditions and boundary conditions. The numerical schemes are defined in \Cref{sec:schemes}.  Existence and uniqueness of the solution are proved in \Cref{sec:exist}. For the convergence of the method, the case of smooth solution and the case of low regularity solution are treated separately in \Cref{sec:conv}. Numerical examples are presented in \Cref{sec:numer}. Conclusions follow.

\section{Continuous problem}
\label{sec:continuous}

We consider two types of networks or hypergraphs (see Fig.~\ref{fig:mesh1d}): 
(i) an edge-network made of one-dimensional edges $\Omega_i$ connected by a set $\Gamma$ of bifurcation nodes (i.e. points) and (ii) a plane-network made of planar polygonal surfaces $\Omega_i$ connected by a set $\Gamma$ of bifurcation nodes (i.e. segments).

\subsection{Notation}

\subsubsection*{Edge-network}

Let $\mathcal{G}$  be a directed graph made of edges $\Omega_i\subset\mathbb{R}$  for $1\leq i\leq N$. The graph $\mathcal{G}$ is embedded in $\mathbb{R}^2$ or $\mathbb{R}^3$.
For each $i$, denote by $\gamma_\mathrm{in}^i$ and $\gamma_\mathrm{out}^i$ the two vertices of $\Omega_i$ and define the values $n_i(\gamma_\mathrm{in}^i) = -1$ and $n_i(\gamma_\mathrm{out}^i) = +1$.
Let $\Gamma$ be the set of bifurcation points; this means that each vertex $\gamma\in\Gamma$ belongs to at least three distinct graph edges.

\subsubsection*{Plane-network}

Let $\mathcal{G}$ be a hypergraph made of planar polygonal surfaces $\Omega_i\subset\mathbb{R}^2$ for $1\leq i\leq N$. The hypergraph $\mathcal{G}$ is embedded in $\mathbb{R}^3$.
Let $\Gamma$ be the set of bifurcation segments (also called interfaces in this work) that result from the intersection of the planar surfaces; this means that each segment $\gamma\in\Gamma$ belongs to at least three distinct planar surfaces. Since each domain $\Omega_i$ is embedded into a two-dimensional plane, we can uniquely define the unit normal outward vector $\bfn_i$ to $\Omega_i$.

\begin{figure}
  \centering
  \includegraphics[height=0.32\textwidth]{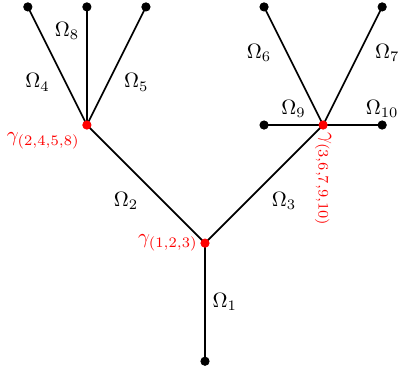}
   \includegraphics[height=0.32\textwidth]{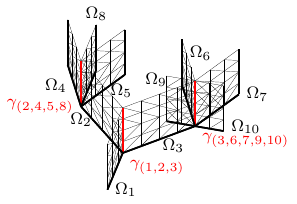}  
  \caption{
    Two types of hypergraphs: An edge-network (left) and a plane-network (right). Bifurcation nodes
    are shown in red: they are points for edge-networks and segments for plane-networks.
  }
  \label{fig:mesh1d}
\end{figure}

\subsection{Model problem}

For each $\gamma\in\Gamma$, we associate the set $I_\gamma$ of indices $i$
such that $\gamma$ belongs to $\overline{\Omega_i}$. In other words, we have
\[
\gamma = \bigcap_{i\in I_\gamma}\overline{\Omega_{i}}.
\]
We also associate to each bifurcation $\gamma$ the set $D_\gamma$ of ordered pairs of indices corresponding to the domains $\overline{\Omega_i}$ that contain $\gamma$:
\[
D_\gamma = \{ (i,j): i < j, \quad i,j \in I_\gamma \}.
\]
We assume that the cardinality of $I_\gamma$ (resp.  of $D_\gamma$) is uniformly bounded.
We denote by $\partial\mathcal{G}$ the union of the boundaries of the domains $\Omega_i$, that do not contain the bifurcations. 
\[
\partial\mathcal{G} = \bigcup_{i=1}^N \partial\Omega_i\setminus \Gamma.
\]
For the edge-network, $\partial\mathcal{G}$ is the set of nodes of the graph that are connected to one edge only. For the plane-network, $\partial\mathcal{G}$ is the union of the line segments that belong each to the boundary of only one planar surface.

Each domain $\Omega_i$ (edge or planar polygonal surface) is associated with a separate coordinate system; 
 this uniquely defines the gradient (or higher derivatives) of a function with respect to that coordinate system.  Nevertheless, for readability, we will use the same notation for the partial derivatives for all the domains $\Omega_i$.  

 We solve an elliptic problem on $\mathcal{G}$. Let $u$
be the exact solution and denote by $u_i = u|_{\Omega_i}$ the restriction of $u$ to $\Omega_i$. The solution $u$ satisfies
\begin{equation}\label{eq:pdei2d}
-\nabla \cdot \bfsigma_i(u_i) = f_i, \quad 
\bfsigma_i(u_i) = \kappa_i \nabla u_i, \quad
\mbox{in}\quad \Omega_i, \quad 1\leq i \leq N,
\end{equation}
with homogeneous Dirichlet boundary condition:
\begin{equation}\label{eq:bc}
u = 0, \quad \mbox{on}\quad\partial\mathcal{G}.
\end{equation}

The functions $f$ and  $\kappa>0$ are given and  we denote $f_i = f|_{\Omega_i}$  and $\kappa_i = \kappa|_{\Omega_i}$.
We assume that $f_i \in L^2(\Omega_i)$ and
$\kappa_i \in L^\infty(\Omega_i)$ for all $1\leq i\leq N$.

To complete the problem posed on $\mathcal{G}$, we impose the following coupling conditions at the bifurcation set $\Gamma$:
\begin{align}\label{eq:coupled_interface2d}
\forall \gamma\in\Gamma,\quad
\sum_{i\in I_\gamma} \bfsigma_i(u_i)|_\gamma\cdot \bfn_i = 0, \\
\forall \gamma\in\Gamma,\quad
\forall i, i' \in I_\gamma, \quad u_i|_\gamma = u_{i'}|_\gamma.\label{eq:coupled_interfaceB}
\end{align}
Condition~\eqref{eq:coupled_interface2d} is known as  Kirchoff's law. For the edge-network, the condition reduces to
\[
\forall \gamma\in\Gamma,\quad
\sum_{i\in I_\gamma} \sigma_i(u_i)|_\gamma\, n_i(\gamma) = 0.
\]
Condition~\eqref{eq:coupled_interfaceB} enforces continuity of the solution at the bifurcation nodes.

Problem~\eqref{eq:pdei2d}-\eqref{eq:coupled_interfaceB} has a unique weak solution in the space $H^1_0(\mathcal{G})$ defined by
\[
H_0^1(\mathcal{G}) = \{ v|_{\Omega_i} = v_i \in H^1(\Omega_i), \,
1\leq i\leq N; \quad v_i|_\gamma = v_{i'}|_\gamma, \, \forall i, i'\in I_\gamma, \, \forall \gamma\in \Gamma; \,
v = 0 \,\, \mbox{on} \, \partial\mathcal{G}\}.
\]
Indeed it is easy to see that the variational formulation is: find $u\in H_0^1(\mathcal{G})$ such that
\begin{equation}
\sum_{i=1}^N \int_{\Omega_i} \kappa_i \nabla u_i \cdot \nabla v_i  = \sum_{i=1}^N \int_{\Omega_i} f v_i, \quad\forall v \in H_0^1(\mathcal{G}). \label{eq:weakformulation}
\end{equation}
Well-posedness is obtained with Lax-Milgram's theorem  thanks to  Poincar\'e's inequality \cite{rupp2022partial}: there is a constant $C_P>0$ such that
\begin{equation}
\forall v \in H_0^1(\mathcal{G}), \quad
\sum_{i=1}^N \Vert v\Vert_{L^2(\Omega_i)}^2 \leq C_P^2
\sum_{i=1}^N \Vert \nabla v \Vert_{L^2(\Omega_i)}^2.
\label{eq:Poincarecontinuous}
\end{equation}

\section{Numerical schemes}
\label{sec:schemes}

\subsection{Preliminaries}

The novel formulation of our proposed schemes stems from a rewriting of the Kirchoff laws using a generalized definition of jumps and averages of functions at the bifurcation nodes. For any scalar function $v$ and vector function $\bfw$, the jump of $v$ and average of $\bfw$ across a bifurcation node that is the intersection of two domains $\overline{\Omega_i}$ and $\overline{\Omega_j}$ are defined by:
\begin{equation}\label{eq:jumpavg22d}
\forall (i,j) \in D_\gamma, \quad
  \jump{v}_{(i, j)}=v|_{\Omega_i} - v|_{\Omega_j},
\quad\quad
\avg{\bfw}_{(i, j)} = \bfw|_{\Omega_i}\cdot\bfn_i - \bfw|_{\Omega_j}\cdot\bfn_j.
\end{equation}
In the case of the edge-network, the average reduces to
\[
\avg{v}_{(i, j)} = v|_{\Omega_i}n_i(\gamma) - v|_{\Omega_j}n_j(\gamma).
\]
In the remainder of the paper, we use the notation $v_i = v|_{\Omega_i}$ for any scalar or vector function $v$ defined on $\mathcal{G}$.

The construction of these generalized jumps and averages leads to the following lemma for the exact solution $u$.  In fact, this is valid for any function $u$ that satisfies the Kirchoff's law \eqref{eq:coupled_interface2d}.
\begin{lemma}\label{lem:modjump}
Let $u, v$ be functions such that $u_i, v_i$ belong to $H^1(\Omega_i)$ for all $1\leq i\leq N$. In addition, assume that $u$ is regular enough to satisfy \eqref{eq:coupled_interface2d}. 
Denote by $\mbox{card}(I_\gamma)$ the cardinality of $I_\gamma$. 
We have
\begin{equation}\label{eq:valuek2}
 \forall \gamma\in\Gamma, \quad
 \sum_{i\in I_\gamma} \bfsigma_i(u_i) \cdot\bfn_i \, v_i = \frac{1}{\mbox{card} (I_\gamma)}
\sum_{(i, j)\in D_\gamma}\avg{\bfsigma(u)}_{(i, j)}\jump{v}_{(i, j)}.
\end{equation}
\end{lemma}
\begin{proof}
We provide the proof for the plane-network, as the case of the edge-network is handled similarly.
Fix $i\in I_\gamma$ and write
\[
\bfsigma_i(u_i) \cdot \bfn_i \, v_i = \frac{1}{\mbox{card} (I_\gamma)}
\sum_{j\in I_\gamma} \bfsigma_i(u_i) \cdot \bfn_i v_i
= \frac{1}{\mbox{card} (I_\gamma)}\bfsigma_i(u_i)\cdot \bfn_i v_i
+ \frac{1}{\mbox{card} (I_\gamma)}
\sum_{j\in I_\gamma, j\neq i} \bfsigma_i(u_i) \cdot\bfn_i v_i.
\]
With the Kirchoff's law \eqref{eq:coupled_interface2d}, we have
\[
\bfsigma_i(u_i) \cdot \bfn_i = -\sum_{j\in I_\gamma, j\neq i}
\bfsigma_j(u_j) \cdot \bfn_j.
\]
This implies
\[
\bfsigma_i(u_i)\cdot \bfn_i v_i
= \frac{1}{\mbox{card} (I_\gamma)} \sum_{j\in I_\gamma, j\neq i} (\bfsigma_i(u_i) \cdot \bfn_i -\bfsigma_j(u_j)\cdot \bfn_j)  v_i.
\]
We now sum over $i$ to obtain
\begin{align*}
\sum_{i\in I_\gamma} \bfsigma_i(u_i) \cdot\bfn_i v_i
=&\frac{1}{\mbox{card} (I_\gamma)} \sum_{i\in I_\gamma}  \sum_{j\in I_\gamma, j < i} 
(\bfsigma_i(u_i) \cdot \bfn_i-\bfsigma_j(u_j)\cdot \bfn_j)  v_i\\
&+ \frac{1}{\mbox{card} (I_\gamma)} \sum_{i\in I_\gamma} \sum_{j\in I_\gamma, j > i} 
(\bfsigma_i(u_i) \cdot \bfn_i-\bfsigma_j(u_j)\cdot \bfn_j  ) v_i
\\
=  &\frac{1}{\mbox{card} (I_\gamma)} \sum_{i\in I_\gamma} \sum_{j\in I_\gamma, j > i}  (\bfsigma_i(u_i) \cdot \bfn_i -\bfsigma_j(u_j) \cdot \bfn_j) (v_i-v_j),
\end{align*}
which is the desired result.
\end{proof}
In what follows, we use the short-hand notation, for any function $w, v$:
\begin{align} \label{eq:short_hand_notation}
\forall \gamma\in\Gamma, \quad (\avg{\bfsigma(w)} \odot \jump{v})|_\gamma
& := \frac{1}{\mbox{card} (I_\gamma)}
\sum_{(i, j)\in D_\gamma} 
\avg{\bfsigma(w)}_{(i, j)}\jump{v}_{(i, j)}, \\
\forall \gamma\in\Gamma, \quad 
(\jump{w}\odot\jump{v})|_\gamma
& := \sum_{(i, j)\in D_\gamma} \jump{w}_{(i, j)}
\jump{v}_{(i, j)}.
\end{align}

\subsection{DG spaces and Poincar\'{e}'s inequality}

We discretize the domains $\Omega_i$ of the hypergraph $\mathcal{G}$ into shape-regular meshes and define a finite dimensional DG space based on the resulting partition.  In the case of the edge-network, each segment $\Omega_i$ is
partitioned into a set of intervals $K$, that forms a mesh $\mathcal{T}_{h,i}$.  In the case of the plane-network, 
each polygonal domain $\Omega_i$  is partitioned into a set of triangular elements  $K$, that forms a mesh $\mathcal{T}_{h,i}$. Let $\mathcal{T}_h$ be the union of the partitions $\mathcal{T}_{h,i}$. 
The mesh size of the domain $\Omega_i$ is $h_i = \max_{K\in\mathcal{T}_{h,i}} \mbox{diam}(K)$  and
the overall mesh size is $h = \max_{1\leq i\leq N} h_i$. 

In the case of the edge-network, for each bifurcation $\gamma$, 
we denote by  $h_\gamma$ the maximum of the lengths of all edges connecting to $\gamma$.
In the case of the plane-network, the meshes are assumed to match at the bifurcation interfaces between two domains. We denote by $\mathcal{F}_h^\gamma$ the resulting partition (made of intervals) of the interface $\gamma\in\Gamma$. 

The DG space is 
\[
V^p_h=\left\{v\in L^2(\mathcal{T}_h): v|_{K} \in \mathbb{P}_p(K), \quad \forall K \in\mathcal{T}_{h,i}, \, \forall 1\leq i\leq N\right\},
\]
where $\mathbb{P}_p(K)$ is the space of polynomials of degree at most $p$ defined on $K$. 
For each mesh $\mathcal{T}_{h,i}$, the set of interior mesh edges for the plane-network (resp. points for the edge-network) is denoted by $\mathcal{F}_{h,i}^{\mathrm{int}}$.
The set of mesh edges for the plane-network (resp. points for the edge-network) that belong to the boundary $\partial\mathcal{G}$ is denoted by $\mathcal{F}_{h,i}^{\partial\mathcal{G}}$.  

For each domain $\Omega_i$, we define the semi-norm
\begin{align*}
\Vert v \Vert_{\DG,i} =
\left( \sum_{K\in\mathcal{T}_{h,i}} \Vert \nabla v_i \Vert_{L^2(K)}^2
 + \sum_{F\in\mathcal{F}_{h,i}^{\mathrm{int}}}
\frac{\eta_{F,i}}{h_F} \Vert [v_i]\Vert_{L^2(F)}^2
+\sum_{F\in\mathcal{F}_{h,i}^{\partial\mathcal{G}}}
\frac{\eta_{F,i}}{h_F} \Vert v_i \Vert_{L^2(F)}^2
\right)^{1/2},
\end{align*}
and the DG norm on $\Omega$ is defined for
the plane-network by:
\begin{align}
\Vert v \Vert_{\DG} = \left(
\sum_{i=1}^N \Vert v \Vert_{\DG,i}^2 
+\sum_{\gamma\in\Gamma} 
\eta_\gamma
\sum_{F\in\mathcal{F}_h^\gamma}
\frac{1}{h_F}
 \sum_{(i, j)\in D_\gamma}
\Vert \jump{v}_{(i, j)}\Vert_{L^2(F)}^2\right)^{1/2},\label{eq:dgnorm2d}
\end{align}
and for the edge network, by:
\begin{align}
\Vert v \Vert_{\DG} = \left(
\sum_{i=1}^N \Vert v \Vert_{\DG,i}^2 
+\sum_{\gamma\in\Gamma} 
\frac{\eta_\gamma}{h_\gamma}
 \sum_{(i, j)\in D_\gamma}
\jump{v}_{(i, j)}^2\right)^{1/2}.\label{eq:dgnorm1d}
\end{align}
In the above, the parameter $h_F$ is the length of the edge $F$ for the plane-network and it is the maximum of the  lengths of the neighboring intervals sharing the point $F$ for the edge-network. The penalty parameters $\eta_{F,i}, \eta_\gamma$ are positive constants, to be specified later. 

We conclude this section by proving a Poincar\'e  inequality for functions in $V_h^p$.  We begin with the enriching map given in \Cref{lemma:enriching_map}; the construction  can be found in many references for classical domains $\Omega \subset \mathbb{R}^d$; see for e.g. \cite{doi:10.1137/S0036142902405217, ern2017finite}. Since our setting is a network of edges or planes, the proof is modified slightly and we provide it for completeness.

We require additional notation. Let $\mathcal{N}_h$ denote the set of all vertices of the mesh $\mathcal{T}_h$ of the network domain $\mathcal{G}$, excluding the vertices  belonging to $\partial \mathcal{G}$. For any $K\in\mathcal{T}_h$, let $\mathcal{N}_K$ denote the set of vertices of $K$. We denote by $\{ \varphi_z \}_{z\in \mathcal{N}_h}$ the corresponding nodal basis of polynomial order $1$. Namely, $\{ \varphi_z \}_{z\in \mathcal{N}_h}$ denotes the basis for the space $V_h^1 \cap H_0^1(\mathcal{G})$. For any vertex $z\in\mathcal{N}_h$, let $\omega_z$ denote the set of mesh elements sharing the node $z$.
The enriching/averaging map $E_h : V_h^1 \rightarrow V_h^1 \cap H^1_0(\mathcal{G})$ is defined  by setting
$$ E_h v = \sum_{z \in \mathcal{N}_h } E_h v(z) \varphi_z , $$
where $E_h v(z)$ is given by
\begin{equation}
E_h v (z) = \frac{1}{\mathrm{card}(\omega_z)} \sum_{K \in \omega_z} v \vert_{K} (z) . 
\end{equation}
Recall that card($\omega_z$) is uniformly bounded with respect to $h$ due to the shape regularity assumption and due to the assumption
that card($I_\gamma$) is uniformly bounded with respect to $h$ for all $\gamma\in \Gamma$.

Finally, throughout the paper, we shall use the standard notation $A \lesssim B$ if there is a generic constant $C$ independent of mesh parameters such that $A \leq C B$.


%
%
%
\begin{lemma}
\label{lemma:enriching_map}
The enriching map $E_h: V_h \rightarrow V_h^1 \cap H_0^1(\mathcal{G})$ satisfies the following properties.  For any $v\in V_h^1$, we have
\begin{equation}\label{eq:L2enrich}
\left(\sum_{i=1}^N \Vert E_h v - v\Vert_{L^2(\Omega_i)}^2\right)^{1/2} \lesssim h
\Vert v \Vert_{\DG},
\end{equation}
and
\begin{equation}\label{eq:gradenrich}
\left(\sum_{i=1}^N \Vert \nabla E_h v \Vert_{L^2(\Omega_i)}^2\right)^{1/2}
\lesssim
\Vert v \Vert_{\mathrm{DG}}.
\end{equation}
\end{lemma}
\begin{proof}
    We give the proof for the plane-network case as a similar and simpler argument can be used for the edge-network.  We need to introduce additional notation. For any $z\in\mathcal{N}_h$, let $\mathcal{F}_z$ denote the set of edges in $\omega_z$ that share the vertex $z$. If the vertex
$z$ does not belong to $\Gamma$, then $\mathcal{F}_z = \mathcal{F}_z^{\text{int}}$ consists only of interior
edges that belong to one of the domains $\Omega_i$. If $z$ lies on one bifurcation $\gamma\in\Gamma$, then
$\mathcal{F}_z$ is the union of disjoint sets of interior edges belonging to $\Omega_i$ for $i\in I_\gamma$, and
of edges that belong to the bifurcation $\gamma$. We write in this case:
$\mathcal{F}_z = \mathcal{F}_z^{\text{int}} \cup \mathcal{F}_z^\gamma$.

We fix $v\in V_h^1$ and $K\in\mathcal{T}_h$. We claim that:
\begin{multline}
\label{eq:L2enrichK}
\Vert E_h v - v\Vert_{L^2(K)}  + h_K \Vert \nabla (E_h v - v)\Vert_{L^2(K)} \\ \lesssim h_K
\sum_{z\in\mathcal{N}_K} \left(
\sum_{F\in \mathcal{F}_z^{\text{int}}}  h_F^{-1/2}\Vert [v]\Vert_{L^2(F)}
+ \sum_{F\in\mathcal{F}_z^{\gamma}} h_F^{-1/2}\sum_{(i,j)\in D_\gamma} \Vert [v]_{(i,j)}\Vert_{L^2(F)}
\right),
\end{multline}
Clearly, a consequence of the above inequalities is the desired result \eqref{eq:L2enrich} and \eqref{eq:gradenrich}. It suffices to show the estimate on $\Vert E_h v - v\Vert_{L^2(K)}$ since the estimate on $\Vert \nabla (E_h v - v)\Vert_{L^2(K)}$ follows from a local inverse inequality. 
By definition of $E_h v$, we  have
\begin{align*}
\Vert E_h v - v\Vert_{L^2(K)} & \lesssim \vert K \vert^{1/2} \sum_{z\in\mathcal{N}_K} \left| E_h v(z) - v|_K(z)\right|\\
& \lesssim \vert K \vert^{1/2} \sum_{z\in\mathcal{N}_K} \sum_{K'\in\omega_z} \left| v|_{K'}(z)-v|_K(z)\right|.
\end{align*}
If $z$ does not belong to $\Gamma$, we have
\[
\sum_{K'\in\omega_z} | v|_{K'}(z)-v|_K(z)| \lesssim \sum_{F\in \mathcal{F}_z^{\text{int}}} |F|^{-1/2} \Vert [v]\Vert_{L^2(F)}.
\]
If $z$ belongs to $\Gamma$, then $z$ belongs to some segment $\gamma$. If we have for instance $K\subset\Omega_{i_0}$,
then
\begin{align*}
\sum_{K'\in\omega_z} | v|_{K'}(z)-v|_K(z)| & \lesssim \sum_{F\in \mathcal{F}_z^{\text{int}}} |F|^{-1/2} \Vert [v]\Vert_{L^2(F)}
+ \sum_{F\in\mathcal{F}_z^{\gamma}} |F|^{-1/2} \sum_{j\in I_\gamma} \Vert [v]_{(i_0,j)}\Vert_{L^2(F)}
\\
& \lesssim \sum_{F\in \mathcal{F}_z^{\text{int}}} |F|^{-1/2} \Vert [v]\Vert_{L^2(F)}
+ \sum_{F\in\mathcal{F}_z^{\gamma}} |F|^{-1/2} \sum_{(i,j)\in D_\gamma} \Vert [v]_{(i,j)}\Vert_{L^2(F)}.
\end{align*}
Therefore, we have in either case:
\[
\Vert E_h v - v\Vert_{L^2(K)} \lesssim 
\sum_{z\in\mathcal{N}_K} \left(
\sum_{F\in \mathcal{F}_z^{\text{int}}} \frac{|K|^{1/2}}{|F|^{1/2}} \Vert [v]\Vert_{L^2(F)}
+ \sum_{F\in\mathcal{F}_z^{\gamma}} \frac{|K|^{1/2}}{|F|^{1/2}} \sum_{(i,j)\in D_\gamma} \Vert [v]_{(i,j)}\Vert_{L^2(F)}
\right).
\]
This implies the desired bound stated in \eqref{eq:L2enrichK}. 
\end{proof}
\begin{lemma}[Stability of $L^2$ projection]
\label{lemma:L2proj}
For any $v\in H^1(\mathcal{T}_h)$, denote by $\pi_h v \in V_h^1$ the $L^2$ projection of $v$ onto $V_h^1$. Then, we have
\begin{equation}\label{eq:L2stab}
\forall v \in H^1(\mathcal{T}_h),\quad
\Vert \pi_h v \Vert_{\DG}
\lesssim \Vert v \Vert_{\DG}.
\end{equation}
\end{lemma}
The proof of the lemma is skipped as it follows standard arguments.

A corollary of \Cref{lemma:enriching_map} and \Cref{lemma:L2proj} is Poincar\'e's inequality in the broken space $H^1(\mathcal{T}_h)$.
\begin{lemma}[Poincar\'e's inequality] \label{lemma:Poincare}
For all $v \in H^1(\mathcal{T}_h)$,
\begin{equation}
  \sum_{i=1}^N \|v\|_{L^2(\Omega_i)}^2 \lesssim \|v\|_{\DG}^2. 
\end{equation}
\end{lemma}
 \begin{proof}
 Fix $v \in H^1(\mathcal{T}_h)$ and denote by $\pi_h v \in V_h^1$ the $L^2$ projection of $v$ onto $V_h^1$. We have
\begin{align*}
\sum_{i=1}^N \|v\|^2_{L^2(\Omega_i)} & \lesssim   \sum_{i=1}^N \|v - \pi_h v\|_{L^2(\Omega_i)}^2
+ \sum_{i=1}^N \|\pi_h v - E_h (\pi_h v)\|^2_{L^2(\Omega_i)} + \sum_{i=1}^N \|E_h (\pi_h v) \|_{L^2(\Omega_i)}^2  \\
& \lesssim  h \|v\|_{\DG}^2 +
 h \|\pi_h v\|^2_{\DG} + \sum_{i=1}^N \Vert E_h(\pi_h v)\Vert_{L^2(\Omega_i)}^2,
\end{align*}
where  we used the approximation  of the $L^2$ projection and
\eqref{eq:L2enrich}.
Next, we apply Poincar\'e's inequality \eqref{eq:Poincarecontinuous} to $E_h(\pi_h v)\in H_0^1(\mathcal{G})$:
\begin{align*}
\sum_{i=1}^N \Vert E_h(\pi_h v)& \Vert_{L^2(\Omega_i)}^2  \lesssim \sum_{i=1}^N \Vert \nabla E_h(\pi_h v)\Vert_{L^2(\Omega_i)}^2
 \lesssim \Vert \pi_h v \Vert_{\DG}^2
 \lesssim \Vert v \Vert_{\DG}^2,
\end{align*}
where we used \eqref{eq:gradenrich} and \eqref{eq:L2stab}. We can then conclude by combining the bounds.
\end{proof}

\subsection{The DG schemes}

The DG approximation of problem~\eqref{eq:pdei2d}-\eqref{eq:coupled_interfaceB} satisfies: find $u_h \in V_h^p$ such that
\begin{align}\label{eq:scheme}
a(u_h, v_h) 
 = \sum_{i=1}^N \int_{\Omega_i} f_i v_h|_{\Omega_i}, \quad\forall v_h\in V_h^p,
\end{align}
where $a(\cdot,\cdot)$ is a bilinear form that is presented below for each type of network for completeness. The forms differ primarily because of the different spatial dimensions. The choice of the parameter $\epsilon_i, \epsilon \in \{-1, 0, 1\}$ in the definition of the bilinear forms, yields the non-symmetric interior penalty Galerkin (NIPG), incomplete interior penalty Galerkin (IIPG), symmetric interior penalty Galerkin (SIPG) versions of the interior penalty DG method. 

\subsubsection*{Bilinear form for edge-network}

\begin{align}
 a(w,v) = &\sum_{i=1}^N a_i(w_i,v_i)
-\sum_{\gamma\in\Gamma} 
\{\sigma(w)\} \odot \jump{v}
-\epsilon \sum_{\gamma\in\Gamma} 
\{\sigma(v)\} \odot \jump{w}
+ \sum_{\gamma\in\Gamma} 
\frac{\eta_\gamma}{h_\gamma} \jump{w} \odot
\jump{v}.
\label{eq:a1Ddef}
\end{align}
The form $a_i(\cdot,\cdot)$ represents the interior penalty DG discretization of the PDE on each segment $\Omega_i$:
\begin{align*}
a_i(w_i,v_i)
= & \sum_{K\in\mathcal{T}_{h,i}} \int_K 
\kappa_i
\frac{\partial w_i}{\partial s} \frac{\partial v_i}{\partial s}
-\sum_{F\in\mathcal{F}_{h,i}^\mathrm{int}}
\{ \kappa_i
\frac{\partial w_i}{\partial s} \}|_F [v_i]|_F
-\epsilon_i \sum_{F\in\mathcal{F}_{h,i}^\mathrm{int}}
\{ \kappa_i
\frac{\partial v_i}{\partial s}\}|_F [w_i]|_F\\
&+ \sum_{F\in\mathcal{F}_{h,i}^\mathrm{int}}
\frac{\eta_{F,i}}{h_F} [w_i]|_F [v_i]|_F
+ \sum_{F\in\mathcal{F}_{h,i}^{\partial\mathcal{G}}} \frac{\eta_{F,i}}{h_F}\,
w_i|_F \, v_i|_F\\
&-\sum_{F\in\mathcal{F}_{h,i}^{\partial\mathcal{G}}} \kappa_i
\frac{\partial w_i}{\partial s}|_F \, n_i(F) \, v_i
-\epsilon_i\sum_{F\in\mathcal{F}_{h,i}^{\partial\mathcal{G}}} \kappa_i
\frac{\partial v_i}{\partial s}|_F \, n_i(F) \, w_i.
\end{align*}

At any interior point $F$, the standard jump operator $[\cdot]|_F$ and average operator $\{ \cdot \}|_F$ used in the form above, are defined by:
\[
v(F^\pm) = \lim_{s\rightarrow 0, s>0} v(F\pm s), \quad
[v]|_F = v(F^{-}) - v(F^{+}), 
\quad
\{ v\}|_F  = \frac12 v(F^{-}) + \frac12 v(F^{+}).
\]

\subsubsection*{Bilinear form for plane-network}

\begin{align}
  a(w,v) = \sum_{i=1}^N a_i(w_i,v_i)
&-\sum_{\gamma\in\Gamma} \sum_{F\in\mathcal{F}_h^\gamma}
\int_F \avg{\bfsigma(w)} \odot \jump{v}
-\epsilon \sum_{\gamma\in\Gamma} 
\sum_{F\in\mathcal{F}_h^\gamma}
\int_F \avg{\bfsigma(v)} \odot \jump{w}
\nonumber\\
&+ \sum_{\gamma\in\Gamma} \eta_\gamma
\sum_{F\in\mathcal{F}_h^\gamma} \frac{1}{h_F}
\int_F  \jump{w} \odot
\jump{v}. \label{eq:a2Ddef}
\end{align}
 For each domain $\Omega_i$, the form
$a_i(\cdot,\cdot)$ is defined by:
\begin{align*}
a_i(w_i,v_i)
= \sum_{K\in\mathcal{T}_{h,i}} \int_K \kappa_i \nabla w_i\cdot\nabla v_i
-\sum_{F\in\mathcal{F}_{h,i}^\mathrm{int}}
\int_F \{\kappa_i \nabla w_i\cdot\bfn_F\} [v_i]\\
-\epsilon_i \sum_{F\in\mathcal{F}_{h,i}^\mathrm{int}}
\int_F \{\kappa_i \nabla v_i\cdot\bfn_F\} [w_i]
+ \sum_{F\in\mathcal{F}_{h,i}^\mathrm{int}}
\frac{\eta_{F,i}}{h_F} \int_F [w_i] [v_i]\\
-\sum_{F\in\mathcal{F}_{h,i}^{\partial\mathcal{G}}}
\int_F \kappa_i \nabla w_i\cdot\bfn_F \, v_i
-\epsilon_i \sum_{F\in\mathcal{F}_{h,i}^{\partial\mathcal{G}}}
\int_F \kappa_i \nabla v_i\cdot\bfn_F \, w_i
+\sum_{F\in\mathcal{F}_{h,i}^{\partial\mathcal{G}}} 
\frac{\eta_{F,i}}{h_F} \int_F w_i \, v_i.
\end{align*}
 To complete the definition of $a_i(\cdot,\cdot)$, we introduce the vectors $\bfn_F$, the jump and average operators $[\cdot], \{\cdot\}$. 
For each interior edge $F\in\mathcal{F}_{h,i}^{\mathrm{int}}$, we fix a unit normal vector to $F$, that is denoted simply by  $\bfn_F$ for readability; this is well defined since the domain $\Omega_i$ is embedded into a plane.
For each boundary edge $F\in \mathcal{F}_{h,i}^{\partial\mathcal{G}}$, the vector $\bfn_F$ is defined to be the outward normal $\bfn_i$ to $\Omega_i$.
We now recall the standard DG jump operator $[\cdot]|_F$ and average operator $\{ \cdot\}|_F$. Fix an edge $F\in \mathcal{F}_{h,i}^{\mathrm{int}}$ and denote by $K_1$ and $K_2$ the two triangles sharing $F$ such that the vector $\bfn_F$ points from $K_1$ into $K_2$. We have
\[
[v]|_F = v|_{K_1}  - v|_{K_2}, \quad
\{v\}|_F = \frac12 v|_{K_1} + \frac12 v|_{K_2}.
\]
\subsection{Derivation of the DG scheme}
\label{sec:deriv}

We present the formal derivation of our scheme only in the case of the plane-network problem, as the case of the edge-network problem is treated similarly.  
We multiply \eqref{eq:pdei2d} by a test function $v\in V_h^p$, use Green's theorem (which is valid if $u|_{\Omega_i} \in H^2(\Omega_i)$ for any $i$) and sum over the domains to obtain 
\[
\sum_{i=1}^N \tilde{a}_i(u_i,v_i)
-\sum_{\gamma\in\Gamma} \sum_{i\in I_\gamma} 
\sum_{F\in\mathcal{F}_h^\gamma}
\int_F \bfsigma_i(u_i)\cdot\bfn_i \, v_i
 = \sum_{i=1}^N \int_{\Omega_i} f_i v_i,
\]
where the form $\tilde{a}_i$ is:
\begin{align*}
\tilde{a}_i(u_i,v_i)
= \sum_{K\in\mathcal{T}_{h,i}} \int_K \kappa_i \nabla u_i\cdot\nabla v_i
-\sum_{F\in\mathcal{F}_{h,i}^\mathrm{int}}
\int_F \kappa_i \nabla u_i\cdot\bfn_F [v_i]|_F
-\sum_{F\in\mathcal{F}_{h,i}^{\partial\mathcal{G}}}
\int_F \kappa_i \nabla u_i\cdot\bfn_F \, v_i.
\end{align*}
With \Cref{lem:modjump}, we rewrite
\[
\sum_{i=1}^N \tilde{a}_i(u_i,v_i)
-\sum_{\gamma\in\Gamma} \sum_{F\in\mathcal{F}_h^\gamma}
\frac{1}{\mathrm{card}(I_\gamma)}
\sum_{(i,j)\in D_\gamma}
\int_F\{ \bfsigma(u)\}_{(i,j)}
[v]_{(i,j)}
 = \sum_{i=1}^N \int_{\Omega_i} f_i v_i.
\]
Using the short-hand notation \eqref{eq:short_hand_notation} and the fact that
$u$ satisfies \eqref{eq:coupled_interfaceB} at each bifurcation, we have
\[
\sum_{i=1}^N \tilde{a}_i(u_i,v_i)
-\sum_{\gamma\in\Gamma}\sum_{F\in\mathcal{F}_h^\gamma}\int_F \{ \bfsigma(u)\} \odot [v]
-\sum_{\gamma\in\Gamma} \sum_{F\in\mathcal{F}_h^\gamma}
\int_F \{ \bfsigma(v)\} \odot [u]
 = \sum_{i=1}^N \int_{\Omega_i} f_i v_i.
\]
For each domain $\Omega_i$, the term $\tilde{a}_i(u_i,v_i)$ leads to the standard interior penalty DG methods.  Indeed, 
on the interior of each domain $\Omega_i$, if $u\in H^2(\Omega_i)$, the jump of the flux
$\kappa_i \nabla u_i \cdot \bfn_F$ across each edge $F$ is zero in the $L^2$ sense.  In addition, the jump of the solution $u_i$ is also zero in the $L^2$ sense. Therefore we apply symmetrization terms and penalty terms to obtain \eqref{eq:scheme}. 

\begin{remark}[Consistency for edge-network]\label{rem:consist}
    In the case of the edge-network, if the coefficient $\kappa|_{\Omega_i}$ is smooth enough (for instance $\kappa_i$ is a constant or $\kappa_i \in W^{1,\infty}(\Omega_i)$), then the exact solution $u|_{\Omega_i}$ belongs to $H^2(\Omega_i)$ for all $1\leq i\leq N$. In that case, the derivation above shows that the DG scheme is consistent:
    \begin{equation}\label{eq:1Dconsistency}
    \forall v_h \in V_h^p, \quad a(u,v_h) = \sum_{i=1}^N \int_{\Omega_i} f_i v_h.
    \end{equation}
\end{remark}

\section{Well-posedness of the DG scheme}
\label{sec:exist}

  We now state and prove coercivity of the form $a(\cdot,\cdot)$ under the following assumption on the parameters $\epsilon_i$ and $\eta_{F,i}$. 
\begin{assumption}\label{hyp1}
Fix $1\leq i\leq N$. 
Assume that on one hand, if  $\epsilon_i = -1$, then $\eta_{F,i} >0$, and on the other hand if $\epsilon_i = +1$ or $0$, then $\eta_{F,i}$ is bounded below by a large enough positive constant.  Assume also that for all bifurcations $\gamma$, the parameter $\eta_\gamma$ is bounded below by a large enough positive constant. 
\end{assumption}

\begin{lemma}\label{lem:coerc}
Under \Cref{hyp1}, we have
\[
\forall v_h\in V_h^p,\quad  
\Vert v_h \Vert_{\DG}^2
\lesssim a(v_h,v_h).
\]
\end{lemma} 
\begin{proof}
Again, here we provide a proof for the plane-network case, the edge-network case being handled in a similar manner.
Using an argument that is standard for the analysis of interior penalty DG methods for elliptic problems \cite{riviere2008discontinuous}, one can show  that if $\eta_{F,i}$ is large enough, we have
\[
\forall v \in V_h^p, \quad  \Vert  v\Vert_{\DG,i}^2 \lesssim a_i(v_i,v_i).
\]
The new contribution is the term involving the bifurcation nodes.  We provide a few details below as the argument is based on standard DG techniques.  We remark that if $w=v$, the last term in \eqref{eq:a2Ddef} appears in the DG norm.  It remains to find a lower bound for the following term
\[
2\sum_{\gamma\in\Gamma} 
\sum_{F\in\mathcal{F}_h^\gamma} 
\int_F \avg{\kappa\nabla v} \odot \jump{v}
\lesssim
\sum_{\gamma\in\Gamma} 
\sum_{(i, j)\in D_\gamma} \sum_{F\in\mathcal{F}_h^\gamma}
\Vert \avg{\kappa\nabla v}_{(i, j)}\Vert_{L^2(F)} \Vert \jump{v}_{(i, j)} \Vert_{L^2(F)} 
\]
\[
\lesssim \Big(\sum_{\gamma\in\Gamma} 
\sum_{(i, j)\in D_\gamma} \sum_{F\in\mathcal{F}_h^\gamma}
h_F^{-1} \Vert \jump{v}_{(i, j)} \Vert_{L^2(F)}^2\Big)^{1/2}\, 
\Big(\sum_{\gamma\in\Gamma} 
\sum_{(i, j)\in D_\gamma} \sum_{F\in\mathcal{F}_h^\gamma}
h_F \Vert \avg{\kappa\nabla v}_{(i, j)} \Vert_{L^2(F)}^2\Big)^{1/2}.
\]
With a discrete trace inequality, we have for a fixed bifurcation $\gamma$, a fixed pair $(i,j)\in D_\gamma$ and a fixed edge $F\in \mathcal{F}_{h,i}^{\mathrm{int}}$
\begin{align*}
h_F \Vert \avg{\kappa\nabla v}_{(i, j)} \Vert_{L^2(F)}^2
&\leq 2 h_F \Vert \kappa_i \nabla v_i \cdot \bfn_i \Vert_{L^2(F)}^2
+ 2 h_F \Vert \kappa_j \nabla v_j \cdot \bfn_j \Vert_{L^2(F)}^2\\
&\lesssim \Vert \nabla v_i \Vert_{L^2(K_F^+)}^2
+ \Vert \nabla v_j \Vert_{L^2(K_F^-)}^2,
\end{align*}
where $K_F^+ \in \mathcal{T}_{h,i}$ and $K_F^-\in\mathcal{T}_{h,j}$ are the mesh elements in the two surfaces $\Omega_i$ and $\Omega_j$ that share $F$.
Since the number of bifurcations is finite  and the cardinality of $D_\gamma$ is uniformly bounded over $\gamma$, we then have
\Bk
\begin{align*}
\sum_{\gamma\in\Gamma} \sum_{(i,j)\in D_\gamma}
\sum_{F\in\mathcal{F}_h^i} h_F \Vert \avg{\kappa\nabla v}_{(i, j)} \Vert_{L^2(F)}^2
\lesssim
\sum_{i=1}^N \sum_{K\in\mathcal{T}_{h,i}} \Vert \nabla v_i \Vert_{L^2(K)}^2.
\end{align*}
Combining the bounds yields coercivity if the penalty parameter $\eta_\gamma$ is large enough.
\end{proof}
Problem~\eqref{eq:scheme} is finite-dimensional; thus the coercivity of the form immediately yields existence and uniqueness of the discrete solution.
\begin{lemma}
Under the conditions of \Cref{lem:coerc}, there exists a unique solution $u_h$ to \eqref{eq:scheme}.
\end{lemma}
With \Cref{lemma:Poincare}, we easily derive a stability bound for the numerical solution.
\begin{lemma}[Stability]
Let $u_h \in V_h^p$ solve \eqref{eq:scheme}. Then, 
\begin{equation}
\|u_h\|_{\DG} \lesssim (\sum_{i=1}^N \|f\|_{L^2(\Omega_i)}^2)^{1/2}. 
\end{equation}
\end{lemma}
\begin{proof}
This is an easy consequence of the coercivity property \Cref{lem:coerc}, Cauchy--Schwarz's inequality and Poincar\'e's inequality \Cref{lemma:Poincare}.
\end{proof}

\section{Convergence} 
\label{sec:conv}
\subsection{A priori error bounds under $H^2$ regularity}

In this section, we assume that the exact solution $u$ is regular enough, namely $u|_{\Omega_i}\in H^2(\Omega_i)$ for $1\leq i\leq N$.  This assumption is certainly valid for the case of the edge-network (see \Cref{rem:consist}) if $\kappa_i$ is smooth enough.  For the plane-network, the solution may not be in $H^2$ on each domain $\Omega_i$ and a low regularity assumption is treated in \Cref{sec:lowregular}.

\begin{lemma}
Assume the exact solution $u|_{\Omega_i}$ belongs to $H^2(\Omega_i)$. Then, Galerkin orthogonality holds:
\[
\forall v_h \in V_h^p, \quad a(u-u_h, v_h) = 0.
\]
\end{lemma}
\begin{proof}
With the local $H^2$ regularity, the derivation of the scheme in \Cref{sec:deriv} is mathematically justified. The exact solution satisfies the scheme, which immediately yields the Galerkin orthogonality.
\end{proof}
Convergence of the method is obtained by deriving optimal a priori error bounds in the DG norm.
\begin{theorem}\label{thm:err_estimate_smooth}
Assume the exact solution $u\in H_0^1(\mathcal{G})$ is such that $u|_{\Omega_i}$ belongs to $H^{s+1}(\Omega_i)$ for $s \geq 1$. 
Then, under \Cref{hyp1}, we have
 \begin{align*}
\Vert u-u_h\Vert_{\mathrm{DG}}
\lesssim h^{\min(p,s)} \left(\sum_{i=1}^N \Vert u_i \Vert_{H^{s+1}(\Omega_i)}^2\right)^{1/2}.
\end{align*}
\end{theorem}
\begin{proof}
The proof follows a standard argument; the only difference is the addition of  the bifurcation terms.  We provide only an outline of the proof.

Let $\tilde{u}\in V_h^p\cap \mathcal{C}(\overline{\Omega})$ be an optimal approximation of $u$. For instance, one may choose for $\tilde{u}$ the Lagrange interpolant of $u$.
We have by coercivity and Galerkin orthogonality
\[
\Vert u_h -\tilde{u}\Vert_{\DG}^2 
\lesssim a(u_h-\tilde{u},u_h-\tilde{u})
= a(u-\tilde{u}, u_h -\tilde{u}).
\]
The first term in the right-hand side is the sum of classical DG formulations; it is handled by a standard DG error analysis. We then obtain
\[
\sum_{i=1}^N a_i(u-\tilde{u}, u_h - \tilde{u})
\lesssim  \Vert u_h - \tilde{u}\Vert_{\DG} \, h^{\min(p,s)} \left(\sum_{i=1}^N \vert u \vert^2_{H^{s+1}(\Omega_i)}\right)^{1/2}.
\]
Continuity of $u$ and $\tilde{u}$ on any $\gamma\in\Gamma$ imply that most bifurcation terms vanish, except for 
\begin{align*}
T & = -\sum_{\gamma\in\Gamma} \sum_{F\in\mathcal{F}_h^\gamma}
\int_F \avg{\bfsigma(u-\tilde{u})} \odot \jump{u_h -\tilde{u}}\\
&=
\sum_{\gamma\in\Gamma} \sum_{F\in\mathcal{F}_h^\gamma}
\frac{1}{\mathrm{card}(I_\gamma)}
\sum_{(i,j)\in D_\gamma}
\int_F \avg{\bfsigma(u-\tilde{u})}_{(i,j)} \jump{u_h -\tilde{u}}_{(i,j)}.
\end{align*}
We have, since $\mathrm{card}(I_\gamma)$ is uniformly bounded
\begin{align*}
T \lesssim &  \Big( 
\sum_{\gamma\in\Gamma} \sum_{F\in\mathcal{F}_h^\gamma}
\sum_{(i,j)\in D_\gamma} \frac{\eta_\gamma}{h_F}
\Vert [u_h-\tilde{u}]_{(i,j)}\Vert_{L^2(F)}^2
\Big)^{1/2} \\
& \quad \times \Big( 
\sum_{\gamma\in\Gamma} \sum_{F\in\mathcal{F}_h^\gamma}
\sum_{(i,j)\in D_\gamma} h_F
\Vert \avg{\bfsigma(u-\tilde{u})}_{(i,j)}\Vert_{L^2(F)}^2
\Big)^{1/2}.
\end{align*}
The first factor is bounded above by $\Vert u_h -\tilde{u}\Vert_{\DG}$. The second factor is treated with trace inequalities and approximation theory, by considering separately each pair $(i,j)\in D_\gamma$ and the mesh elements $K^-\in \mathcal{T}_{h,i}$ and
$K^+\in\mathcal{T}_{h,j}$ that share the face $F\in \mathcal{F}_h^\gamma$.
\end{proof}

\subsection{A priori error bounds under low regularity for plane-networks}
\label{sec:lowregular}

Since we do not know a priori that $u|_{\Omega_i} \in H^2(\Omega_i)$ for the plane-network case, Galerkin orthogonality does not hold and therefore the error analysis becomes non-standard.  In this section, we fix $0<s<1$ and we work with the following space 
\begin{equation}
V^{1+s} = \{ v \in H^1_0(\mathcal{G}):  \; v_i\in H^{1+s} (\Omega_i), \; \; \nabla \cdot \bm \sigma_i(v_i) \in L^2(\Omega_i) ,\; \; 1 \leq i \leq N\}. 
\end{equation}
In what follows, a weak meaning for the normal trace of the fluxes $\bm \sigma_i (v_i) $ for $v \in V^{1+s}$ is defined following \cite{ern2021quasi}. This will allow us to show a weak notion of consistency by extending the form $a(\cdot,\cdot)$ to the space $V^{1+s}\times V_h^p$. 

For $K  \in \mathcal{T}_h$, denote by $\bm n_K$ the unit normal vector outward of $K$ and define $\mathcal{F}_K$ as the set of edges belonging to $\partial K$. Define $\rho = 4 /(2-2s)$ and let $\rho'$ be its conjugate, $1/\rho + 1/\rho' = 1$.  Observe that $u_i \in H^{1+s}(K)$ implies that $\nabla u_i \in H^{s}(K)$, and by Sobolev embedding~\cite[Theorem 2.31]{Ern:booki}, this implies that $\nabla u_i \in L^{\rho}(K)^2$. Hence if $\kappa_i \in L^{\infty}(\Omega_i)$, $\bm \sigma_i (u_i) \in L^\rho (\Omega_i)^2$. We consider the lifting operator~\cite[Lemma 3.1]{ern2021quasi}, that satisfies for
all $K\in\mathcal{T}_h$ and $F\in\mathcal{F}_K$:
\begin{equation}
L^K_F: W^{\frac{1}{\rho}, \rho'}(F) \rightarrow W^{1, \rho'}(K), \quad  \gamma_0 (L_F^K(\phi))\vert_{\partial K}  = \begin{cases} \phi & \mathrm{on} \,\,\, F, 
\\ 
0 & \mathrm{otherwise} 
\end{cases},
\end{equation}
where $\gamma_0: W^{1,\rho'}(K) \rightarrow W^{1-1/\rho',\rho'}(\partial K)$ is the trace operator. 
This lifting operator is utilized to give weak meaning to the normal traces. For each $K \in \mathcal{T}_h$, the normal trace of $\bm \sigma_i(u_i)$ on an edge $F \in \mathcal{F}_K$ is defined as a member of the dual space of $W^{1/\rho, \rho'}(F)$ as follows
\begin{equation}
\langle (\bm \sigma_i(u_i)\cdot \bm n_K)_{\vert F}, \phi\rangle_{F} := \int_K  (\bm \sigma_i(u_i) \cdot \nabla L_F^K(\phi) + (\nabla \cdot \bm \sigma_i( u_i)) \, L_F^K(\phi)  ), \quad \forall\phi \in  W^{1/\rho, \rho'}(F). \label{eq:def_duality_pair}
\end{equation}
The above is well defined since $\bm \sigma_i(u_i) \in L^{\rho}(K)$, $\nabla \cdot \bm \sigma_i(u_i)\in L^2(K)$ and $L_F^K(\phi) \in L^2(K)$ because $W^{1,\rho'} (K) \hookrightarrow L^2(K)$ by  Sobolev embedding~\cite[Theorem 2.31]{Ern:booki}. We will also require the following space: for $1 \leq r < \infty$ 
\begin{align}
  W^{1,r}_0(\mathcal{G}) &=   \{ v\vert_{\Omega_i} =  v_i \in W^{1,r}(\Omega_i), \; 
1\leq i\leq N \} \cap H^1_0(\mathcal{G}). \label{eq:def_sobolev_graph}
\end{align}


The next step is to consider a suitable extension of the form $a$ to the space $V^{1+s} \times V_h^p$. Let $\mathcal{T}_F$ denote the elements $K$ sharing the edge $F$ and let $\epsilon_{K,F} = (\bm n_K \vert_F) \cdot \bm n_F$. Clearly, $\epsilon_{K,F}$ takes the value $\pm 1$ depending on the orientation of $\bm n_K$ with respect to $\bm n_F$. Observe that if $F \in \mathcal{F}_{h}^\gamma$  then  the cardinality of $\mathcal{T}_F$ is equal to  $\mathrm{card}(I_\gamma)$.
For each edge that is interior to a domain, $\  F \in \cup_{i = 1}^{N} \mathcal{F}_{h,i}^{\mathrm{int}},$ define the form $\tau_F: V^{1+s} \times V_h^p \rightarrow \mathbb{R}$ as follows
\begin{equation}
\tau_F (v, w_h) = \frac12
\sum_{K \in \mathcal{T}_F} \epsilon_{K,F} \langle (\bm \sigma(v)\vert_K \cdot \bm n_K) \vert_F, [w_h]  \rangle_F.  
\end{equation} 
For each boundary edge $F\in \cup_{i=1}^N \mathcal{F}_{h,i}^{\partial\mathcal{G}}$, we simply define
\begin{equation}
\tau_F (v, w_h) = 
\sum_{K \in \mathcal{T}_F}  \langle (\bm \sigma(v)\vert_K \cdot \bm n_K) \vert_F, w_h  \rangle_F.  
\end{equation} 
For each edge $F$ belonging to a bifurcation segment, namely $F \in \mathcal{F}_h^\gamma$, we modify the above definition since  $\mathrm{card}(I_\gamma) > 2$ and we define 
\begin{equation}
 \tau_F^\gamma (v, w_h) = \frac{1}{\mathrm{card}(I_{\gamma})}
 \sum_{(i,j) \in D_\gamma} \sum_{K \in \mathcal{T}_F \cap (\Omega_i \cup \Omega_j) } \epsilon_{K,F} \langle (\bm \sigma (v)\vert_K \cdot \bm n_K) \vert_F, [w_h]_{(i,j)}  \rangle_F.  
\end{equation} 
The extension of the form $a$ can now be defined as $\mathcal{A}: V^{1+s} \times V_h^p \rightarrow \mathbb{R}$: 
\begin{equation}
\mathcal{A}(v,w_h) = \sum_{i=1}^N \sum_{K \in \mathcal{T}_{h,i}} \int_K \bm \sigma(v) \cdot \nabla w_h   - \sum_{i=1}^N \sum_{F \in \mathcal{F}_{h,i}^{\mathrm{int}} \cup \mathcal{F}_{h,i}^{\partial \mathcal{G}}} \tau_F(v,w_h) -  \sum_{\gamma \in \Gamma} \sum_{F \in \mathcal{F}_h^\gamma } \tau_F^\gamma (v, w_h).  \label{eq:extended_form_A}
\end{equation}
\begin{lemma}\label{lemma:froms_agree_discrete}
    For any $v , w  \in V_h^p$, the following identities hold 
\begin{align}\label{eq:identity_eta_interior}
\forall 1\leq i\leq N, \quad \sum_{F \in \mathcal{F}_{h,i}^{\mathrm{int}}} \tau_F(v ,w ) 
&= 
\sum_{F \in \mathcal{F}_{h,i}^{\mathrm{int}}}\int_F \{\bfsigma(v) \cdot\bfn_F\} [w], \\ 
\label{eq:identity_eta_bdy}
\forall 1\leq i\leq N, \quad \sum_{F \in \mathcal{F}_{h,i}^{\partial \mathcal{G}}} \tau_F(v ,w ) 
&= 
\sum_{F \in \mathcal{F}_{h,i}^{\partial \mathcal{G}}} \int_F \bfsigma(v) \cdot\bfn_F\, w, \\ 
\forall \gamma\in\Gamma, \quad \sum_{F \in \mathcal{F}_h^\gamma } \tau_F^\gamma (v, w)  
 & = \sum_{F\in\mathcal{F}_h^\gamma}
\int_F \avg{\bfsigma(v)} \odot \jump{w} .  \label{eq:identity_eta_interface}
\end{align}
\end{lemma} 
\begin{proof}
The proof of \eqref{eq:identity_eta_interior}-\eqref{eq:identity_eta_bdy} is from~\cite[Lemma 3.3]{ern2021quasi}. We provide the similar proof for \eqref{eq:identity_eta_interface} for completeness. Since $v \in V_h^p$, $\bm \sigma(v)$ is smooth when restricted to a single element $K \in \mathcal{T}_h$ in $\Omega_i$. Using Green's theorem in \eqref{eq:def_duality_pair}, we obtain that  
\begin{align}
    \langle (\bm \sigma_i(v_i)\vert_K \cdot \bm n_K) \vert_F, [w_h]_{(i,j)}  \rangle_F  &   = \int_{\partial K} \bm \sigma (v)\vert_{K} \cdot \bm n_K  \gamma_0(L_F^K([w_h]_{(i,j)})) \\ 
    & =  \int_{F} \bm \sigma (v)\vert_{K}  \cdot \bm n_K [w_h]_{(i,j)},  \nonumber
\end{align}
where the last equality holds since  the trace of $L_F^K([w_h])$ is nonzero only on $F$. Recalling that $\epsilon_{K,F} = (\bm n_K) \vert_F \cdot \bm n_F$, we obtain that 
\begin{align}
\sum_{K \in \mathcal{T}_F \cap (\Omega_i \cup \Omega_j)} \epsilon_{K,F} \langle (\bm \sigma(v)\vert_K \cdot \bm n_K) \vert_F, [w_h]_{(i,j)}  \rangle_F  =  \int_F \{\bm \sigma (v) \}_{(i,j)} [w_h]_{(i,j)}. 
\end{align}
Summing the above over $(i,j) \in D_\gamma$, dividing by $\mathrm{card}(I_\gamma)$ and recalling the notation \eqref{eq:short_hand_notation} yields the result. 
\end{proof}

Similar to the proofs in \cite{ern2021quasi}, we  utilize the mollification operators constructed in \cite{ern2016mollification,schoberl2008} for any $0<1<\delta$: 
\begin{equation}
\mathcal K_{\delta,i}^\mathfrak{d}: L^1(\Omega_i)^2 \rightarrow C^{\infty}(\overline{\Omega_i})^2,  \quad \mathcal K_{\delta,i}^\mathfrak{b} : L^1(\Omega_i) \rightarrow C^{\infty}(\overline{\Omega_i}), 
\end{equation}
which satisfy the commutativity property for all $ \bm \tau \in L^1(\Omega_i)^2$ with $\nabla \cdot \bm \tau \in L^1(\Omega_i)$
\begin{equation}
    \nabla \cdot \mathcal K_{\delta,i}^\mathfrak{d} (\bm \tau) = \mathcal K_{\delta,i}^\mathfrak{b}(\nabla \cdot \bm \tau). \label{eq:commutative_prop_moll} 
\end{equation} 
The superscript $\mathfrak{d}$ in the operator $\mathcal K_{\delta,i}^\mathfrak{d}$ refers to the divergence whereas the superscript $\mathfrak{b}$ in the operator $\mathcal K_{\delta,i}^\mathfrak{b}$ refers to the broken space.
We extend these operators to the hypergraph by denoting $\mathcal{K}^{\mathfrak{b}}_{\delta}$ and $\mathcal{K}^{\mathfrak{d}}_{\delta}$ the operators such as:
$$ \mathcal{K}^\mathfrak{d}_{\delta} (\bm \tau ) \vert_{\Omega_i} =  \mathcal{K}^\mathfrak{d}_{\delta,i} (\bm \tau \vert_{\Omega_i}), \quad  \mathcal{K}^\mathfrak{b}_{\delta} (w) \vert_{\Omega_i} =  \mathcal{K}^\mathfrak{b}_{\delta,i} (w \vert_{\Omega_i}) , \quad 1 \leq i \leq N. $$ 

From \cite[Corollary 3.4]{ern2016mollification}, we find that for $\bm \tau \in L^\rho(\mathcal{G})^2$ and for any $w \in L^2(\mathcal{G})$
\begin{align}
\lim_{\delta \rightarrow 0} \| \mathcal{K}_{\delta}^\mathfrak{d} \bm \tau - \bm \tau \|_{L^\rho(\mathcal{G})} = 0  , \quad \lim_{\delta \rightarrow 0} \| \mathcal{K}_{\delta}^\mathfrak{b} (w) - w \|_{L^2(\mathcal{G})} = 0 .  \label{eq:convergence_mollifiers}
\end{align}

In the following lemma, we assume that the weak formulation \eqref{eq:weakformulation} satisfied by the weak solution $u$ is also true for a weaker test space: see \eqref{eq:weak_Kirchhoff_condition}. 

\begin{lemma}\label{lemma:weak_Kirchhoff_condition}
Suppose the weak solution $u$ to \eqref{eq:weakformulation} is such that $u \in V^{1+s}$ and that it satisfies 
\begin{equation} 
\label{eq:weak_Kirchhoff_condition}
\sum_{i=1}^N (\bm \sigma_i(u_i), \nabla v_i)_{\Omega_i} = \sum_{i=1}^N (f_i, v_i)_{\Omega_i}, \quad \forall v \in W^{1,\rho'}_0(\mathcal{G}). 
\end{equation}
%
Then, for any $\gamma \in \Gamma$, $ F \in \mathcal{F}_h^\gamma$ and  $w_h \in V_h^p$, we have  
\begin{equation}
\lim_{\delta \rightarrow 0 }  \sum_{i \in I_\gamma} \int_F (\mathcal{K}_{\delta}^\mathfrak{d} \bm \sigma (u)) \vert_{\Omega_i} \cdot \bfn_i   \, w_h \vert_{\Omega_i} =  \lim_{\delta \rightarrow 0} \int_F \{ \mathcal{K}_{\delta}^\mathfrak{d} (\bm \sigma (u)) \} \odot [w_h].
\label{eq:smoothgener}
\end{equation}
\end{lemma} 
\begin{proof}
First, we  show that for any $\gamma \in \Gamma$, $F \in \mathcal{F}_h^\gamma$ and  $\phi \in C^\infty(F)$, the following holds 
\begin{equation}
\lim_{\delta \rightarrow 0 }  \sum_{i \in I_\gamma} \int_F (\mathcal{K}_{\delta}^\mathfrak{d} \bm \sigma (u)) \vert_{\Omega_i} \cdot \bfn_i \phi  = 0. \label{eq:weak_Kirchhoff_key} 
\end{equation}
To show \eqref{eq:weak_Kirchhoff_key}, fix $F \in \mathcal{F}_h^\gamma$ and denote by $K_i \in \mathcal{T}_h$ the element that belongs to $\Omega_i$ and shares the edge $F$. In other words, $K_i = \mathcal{T}_F \cap \mathcal{T}_{h,i}$.  Since the trace of $L_F^{K_i}(\phi)$ is non zero only on $F$, we write 
\begin{align}
\int_F (\mathcal{K}_{\delta}^\mathfrak{d} \bm \sigma (u)) \vert_{\Omega_i} \cdot \bfn_i \, \phi =  \int_{\partial K_i} (\mathcal{K}_{\delta}^\mathfrak{d} \bm \sigma (u)) \vert_{\Omega_i} \cdot \bfn_i \, L_F^{K_i}(\phi). 
\end{align}
Since $\mathcal K_\delta^\mathfrak{d} \bm \sigma (u)$ is smooth in $\Omega_i$, we can apply Green's theorem to obtain 
\begin{align}
\int_F (\mathcal{K}_{\delta}^\mathfrak{d} \bm \sigma (u)) \vert_{\Omega_i} \cdot \bfn_i \, \phi =  \int_{ K_i} \nabla \cdot (\mathcal{K}_{\delta}^\mathfrak{d} \bm \sigma (u)) \, L_F^{K_i}(\phi) + \int_{K_i}  (\mathcal{K}_{\delta}^\mathfrak{d} \bm \sigma (u)) \cdot \nabla L_F^{K_i}(\phi). 
\end{align}
Owing to the commutativity property \eqref{eq:commutative_prop_moll}, the convergence of the mollification operators \eqref{eq:convergence_mollifiers} and the fact that $L_F^{K_i} (\phi) \in W^{1,\rho'}(K_i)$, we obtain
\begin{align} \label{eq:step_0_kirk}
\lim_{\delta \rightarrow 0} \int_F (\mathcal{K}_{\delta}^\mathfrak{d} \bm \sigma (u)) \vert_{\Omega_i} \cdot \bfn_i \, \phi  = \int_{ K_i} \nabla \cdot ( \bm \sigma (u)) \, L_F^{K_i}(\phi) + \int_{K_i}   \bm \sigma (u) \cdot \nabla L_F^{K_i}(\phi) < \infty. 
\end{align}
Define $L_F^{\Omega_i} (\phi)$ as the extension by zero of $L_F^{K_i}(\phi)$ in $\Omega_i$. That is, $L_F^{\Omega_i}(\phi) = 0$ in $K$ if $K \neq K_i $ and $L_F^{\Omega_i}(\phi) = L_F^{K_i}(\phi) $ in $K_i$.  Since $L_F^{K_i} (\phi) \vert_{\partial K \backslash F} = 0$, the jump of $L_F^{\Omega_i}(\phi)$ evaluates to zero on the interior edges of $\Omega_i$.  Therefore, $L_F^{\Omega_i} (\phi) \in W^{1,\rho'}(\Omega_i)$. Extending this to the hypergraph, we define $L_F^{\mathcal{G}}(\phi)$ as follows 
\begin{align}
\forall 1\leq i\leq N, \quad
L_F^{\mathcal{G}}(\phi)|_{\Omega_i} = 
\begin{cases}
L_F^{\Omega_i} (\phi) &  \;\; i \in I_{\gamma}, \\ 
0 &  \;\;  i \notin I_{\gamma}.
\end{cases} 
\end{align}
It is easy to see that $L_F^{\mathcal{G}}(\phi) \in W_0^{1,\rho'}(\mathcal{G})$, see \eqref{eq:def_sobolev_graph}. Testing \eqref{eq:weak_Kirchhoff_condition} with  $ L_F^{\mathcal{G}}(\phi)$,  we obtain that 
\begin{align}
\sum_{i \in I_\gamma}  \int_{K_i} \bm \sigma_i(u_i)  \cdot \nabla L_F^{K_i}(\phi)  = \sum_{i \in I_{\gamma}}  \int_{K_i} f_i \,   L_F^{K_i}(\phi). 
\end{align}
With the above identity, \eqref{eq:step_0_kirk} yields
\begin{align*} 
\lim_{\delta \rightarrow 0} \sum_{i\in I_\gamma} \int_F (\mathcal{K}_{\delta}^\mathfrak{d} \bm \sigma (u)) \vert_{\Omega_i} \cdot  \bfn_i \, \phi  = \sum_{i\in I_\gamma} \int_{ K_i} \nabla \cdot ( \bm \sigma (u)) \,  L_F^{K_i}(\phi) + 
\sum_{i \in I_{\gamma}}  \int_{K_i} f_i   \, L_F^{K_i}(\phi).
\end{align*}
We then conclude and obtain \eqref{eq:weak_Kirchhoff_key} by noting that $- \nabla \cdot \bm \sigma (u ) \vert_{\Omega_i}= f_i $ in $L^2(\Omega_i)$.

Next, we prove \eqref{eq:smoothgener}.  The proof follows the argument of \Cref{lem:modjump}. Fix $w_h \in V_h^p$ and fix $F\in \mathcal{F}_h^\gamma$.  We note that for a fixed $i\in I_\gamma$ and with the short-hand notation $m_\gamma = \mathrm{card}(I_\gamma)$ for readibility
\[
(\mathcal{K}_{\delta}^\mathfrak{d} \bm \sigma (u)) \vert_{\Omega_i} \cdot  \bfn_i \, w_h|_{\Omega_i} 
= \frac{1}{m_\gamma}
(\mathcal{K}_{\delta}^\mathfrak{d} \bm \sigma (u)) \vert_{\Omega_i} \cdot  \bfn_i \, w_h|_{\Omega_i} 
+ \frac{1}{m_\gamma}
\sum_{j\in I_\gamma, j\neq i}
(\mathcal{K}_{\delta}^\mathfrak{d} \bm \sigma (u)) \vert_{\Omega_i} \cdot  \bfn_i \, w_h|_{\Omega_i}.
\]
Integrating over $F$ and taking the limit $\delta\rightarrow 0$, we have
\begin{multline*}
\lim_{\delta \rightarrow 0} \int_F (\mathcal{K}_{\delta}^\mathfrak{d} \bm \sigma (u)) \vert_{\Omega_i} \cdot  \bfn_i \, w_h|_{\Omega_i} \\
= \frac{1}{m_\gamma}
\lim_{\delta \rightarrow 0}  \int_F(\mathcal{K}_{\delta}^\mathfrak{d} \bm \sigma (u)) \vert_{\Omega_i} \cdot  \bfn_i \, w_h|_{\Omega_i} 
+ \frac{1}{m_\gamma} \lim_{\delta \rightarrow 0} 
\sum_{j\in I_\gamma, j\neq i}
\int_F (\mathcal{K}_{\delta}^\mathfrak{d} \bm \sigma (u)) \vert_{\Omega_i} \cdot  \bfn_i \, w_h|_{\Omega_i}.
\end{multline*}
With \eqref{eq:weak_Kirchhoff_key}, we write  with the choice of $\phi = w_h|_{\Omega_i}$
\[
\lim_{\delta \rightarrow 0} \int_F(\mathcal{K}_{\delta}^\mathfrak{d} \bm \sigma (u)) \vert_{\Omega_i} \cdot  \bfn_i \, w_h|_{\Omega_i}
= - \lim_{\delta \rightarrow 0}
\sum_{j\in I_\gamma, j\neq i}
\int_F (\mathcal{K}_{\delta}^\mathfrak{d} \bm \sigma (u)) \vert_{\Omega_j} \cdot  \bfn_j \, w_h|_{\Omega_i}.
\]
Therefore combining with the above, we have
\begin{equation}
\lim_{\delta \rightarrow 0} \int_F (\mathcal{K}_{\delta}^\mathfrak{d} \bm \sigma (u)) \vert_{\Omega_i} \cdot  \bfn_i \, w_h|_{\Omega_i}
=
\frac{1}{m_\gamma} \lim_{\delta \rightarrow 0} 
\sum_{j\in I_\gamma, j\neq i}
\int_F \left((\mathcal{K}_{\delta}^\mathfrak{d} \bm \sigma (u)) \vert_{\Omega_i} \cdot  \bfn_i 
-  (\mathcal{K}_{\delta}^\mathfrak{d} \bm \sigma (u)) \vert_{\Omega_j} \cdot  \bfn_j \right)  w_h|_{\Omega_i}.
\label{eq:int1}
\end{equation}
Clearly we have
\begin{align*}
\frac{1}{m_\gamma} 
\sum_{i\in I_\gamma}
\sum_{j\in I_\gamma, j\neq i}
\int_F \left((\mathcal{K}_{\delta}^\mathfrak{d} \bm \sigma (u)) \vert_{\Omega_i} \cdot  \bfn_i 
-  (\mathcal{K}_{\delta}^\mathfrak{d} \bm \sigma (u)) \vert_{\Omega_j} \cdot  \bfn_j \right)  w_h|_{\Omega_i}\\
= \frac{1}{m_\gamma} 
\sum_{(i,j)\in D_\gamma}
\left((\mathcal{K}_{\delta}^\mathfrak{d} \bm \sigma (u)) \vert_{\Omega_i} \cdot  \bfn_i 
-  (\mathcal{K}_{\delta}^\mathfrak{d} \bm \sigma (u)) \vert_{\Omega_j} \cdot  \bfn_j \right) 
(w_h|_{\Omega_i} - w_h|_{\Omega_j}).
\end{align*}
With this and summing \eqref{eq:int1} over $i\in I_\gamma$, we obtain the desired result.
\end{proof}

\begin{remark}
The property \eqref{eq:weak_Kirchhoff_condition} holds true if the space
$\mathcal{C}^\infty(\cup_i \Omega_i) \cap H_0^1(\mathcal{G})$ is dense in $W^{1,r}_0(\mathcal{G})$ for $1 < r < 2$. 
\end{remark}

\begin{lemma}[Weak consistency]\label{lemma:weak_consistency}
Assume that the solution $u$ to \eqref{eq:pdei2d} is such that $u \in V^{1+s}$ and that \eqref{eq:weak_Kirchhoff_condition} holds. Then, 
\begin{align}
    \mathcal{A}(u,w_h) = \sum_{i=1}^N \int_{\Omega_i} f_i w_h\vert_{\Omega_i}, 
\quad \forall w_h \in V_h^p. 
\end{align}
\end{lemma}
\begin{proof}
Following \cite{ern2021quasi}, we approximate the forms  $\tau_F$ and $\tau^\gamma_F$ by introducing the mollification operator $\mathcal{K}_\delta^\mathfrak{d}$ in the forms:
\begin{align*}
\tau_{F,\delta} (u, w_h) & = \frac12
\sum_{K \in \mathcal{T}_F} \epsilon_{K,F}  \langle (\mathcal{K}_{\delta}^\mathfrak{d} 
 \bm \sigma(u)\vert_K \cdot \bm n_K) \vert_F, [w_h]  \rangle_F  , \quad F \in  \cup_{i=1}^N, \mathcal{F}_{h,i}^\mathrm{int},\\
 \tau_{F,\delta} (u, w_h) & = 
\sum_{K \in \mathcal{T}_F}  \langle (\mathcal{K}_{\delta}^\mathfrak{d} 
 \bm \sigma(u)\vert_K \cdot \bm n_K) \vert_F, w_h  \rangle_F  , \quad F \in  \cup_{i=1}^N \mathcal{F}_{h,i}^\mathrm{\partial\mathcal{G}},
 \end{align*}
 and for any $\gamma\in\Gamma$:
 \begin{align*}
 \tau^\gamma_{F, \delta} (u,w_h) & = \frac{1}{\mathrm{card}(I_{\gamma})}
 \sum_{(i,j) \in D_\gamma} \sum_{K \in \mathcal{T}_F \cap (\Omega_i \cup \Omega_j) } \epsilon_{K,F} \langle (\mathcal{K}_{\delta}^\mathfrak{d} 
 \bm \sigma (u)\vert_K \cdot \bm n_K) \vert_F, [w_h]_{(i,j)}  \rangle_F, \quad
 F \in \mathcal{F}_h^\gamma.
\end{align*}
With \eqref{eq:convergence_mollifiers}, for any polynomial function $\xi_h$ defined on $F$, we have 
\begin{align*}
 \lim_{\delta \rightarrow 0} \langle (\mathcal{K}_{\delta}^\mathfrak{d} 
 \bm \sigma(u)\vert_K \cdot \bm n_K) \vert_F, \xi_h  \rangle_F  & 
 =  \lim_{\delta \rightarrow 0} \left( \int_{K} \nabla \cdot ( \mathcal{K}_{\delta}^\mathfrak{b} 
\bm  \sigma (u))  L_F^{K}(\xi_h) + \int_{K}  \mathcal{K}_{\delta}^\mathfrak{d} 
\bm  \sigma (u) \cdot \nabla L_F^{K}(\xi_h) \right) \\ 
 & = \int_{K} \nabla \cdot ( \bm \sigma (u))  L_F^{K}(\xi_h) + \int_{K} \bm \sigma (u) \cdot \nabla L_F^{K}(\xi_h)  \\ 
 & =  \langle 
( \bm \sigma(u)\vert_K \cdot \bm n_K) \vert_F, \xi_h  \rangle_F.
\end{align*}
This implies
\begin{align}
\lim_{\delta \rightarrow 0 }  \tau_{F,\delta}(u,w_h) = \tau_F(u,w_h), \quad \lim_{\delta \rightarrow 0 }   \tau^\gamma_{F,\delta}(u,w_h) = \tau^\gamma_F(u,w_h). \label{eq:weak_consistency_0}
\end{align}
Therefore we have
\begin{align*}
\sum_{i=1}^N \sum_{F \in \mathcal{F}_{h,i}^{\mathrm{int}} \cup \mathcal{F}_{h,i}^{\partial \mathcal{G}}}  & \tau_{F} (u,w_h )  +  \sum_{\gamma \in \Gamma} \sum_{F \in \mathcal{F}_h^\gamma } \tau^\gamma_{F} (u, w_h) \\ & =  \lim_{\delta \rightarrow 0}    \sum_{i=1}^N \sum_{F \in \mathcal{F}_{h,i}^{\mathrm{int}} \cup \mathcal{F}_{h,i}^{\partial\mathcal{G}}} \tau_{F,\delta} (u,w_h )  + \lim_{\delta \rightarrow 0 }\sum_{\gamma 
\in \Gamma} \sum_{F \in \mathcal{F}_h^\gamma } \tau^\gamma_{F,\delta} (u, w_h). 
\end{align*}
Since $\mathcal{K}_{\delta}^\mathfrak{d}(u)$ is smooth when restricted to one domain, an argument similar to \Cref{lemma:froms_agree_discrete} yields 
\begin{align*}
\sum_{i=1}^N \sum_{F \in \mathcal{F}_{h,i}^{\mathrm{int}} \cup \mathcal{F}_{h,i}^{\partial \mathcal{G}}} \tau_{F,\delta} (u, w_h )
= &
\sum_{i=1}^N 
\sum_{F \in \mathcal{F}_{h,i}^{\mathrm{int}}\cup \mathcal{F}_{h,i}^{\partial\mathcal{G}}} \int_{F } \{ \mathcal{K}_{\delta}^\mathfrak{d} 
 \bfsigma(u_i) \cdot\bfn_F\} [w_h],\\
\sum_{\gamma \in \Gamma} \sum_{F \in \mathcal{F}_h^\gamma } \tau^\gamma_{F,\delta} (u, w_h)   
=  &\sum_{\gamma \in \Gamma} \sum_{F\in\mathcal{F}_h^\gamma}
\int_F \{\mathcal{K}_{\delta}^\mathfrak{d} 
 \bfsigma(u)\} \odot \jump{w_h}. 
\end{align*} 
Therefore, we have with \Cref{lemma:weak_Kirchhoff_condition}
\begin{align*}
&\sum_{i=1}^N \sum_{F \in \mathcal{F}_{h,i}^{\mathrm{int}} \cup \mathcal{F}_{h,i}^{\partial\mathcal{G}}}   \tau_{F} (u,w_h )  +  \sum_{\gamma \in \Gamma} \sum_{F \in \mathcal{F}_h^\gamma } \tau^\gamma_{F} (u, w_h) \\
&= \lim_{\delta \rightarrow 0}
\sum_{i=1}^N 
\sum_{F \in \mathcal{F}_{h,i}^{\mathrm{int}}\cup \mathcal{F}_{h,i}^{\partial}} \int_{F } \{ \mathcal{K}_{\delta}^\mathfrak{d} 
 \bfsigma(u_i) \cdot\bfn_F\} [w_h]
 + \lim_{\delta \rightarrow 0}
 \sum_{\gamma \in \Gamma} \sum_{F\in\mathcal{F}_h^\gamma}
\int_F \{\mathcal{K}_{\delta}^\mathfrak{d} 
 \bfsigma(u)\} \odot \jump{w_h}\\
 & = \lim_{\delta \rightarrow 0}
\sum_{i=1}^N 
\sum_{F \in \mathcal{F}_{h,i}^{\mathrm{int}}\cup \mathcal{F}_{h,i}^{\partial\mathcal{G}}} \int_{F } \{ \mathcal{K}_{\delta}^\mathfrak{d} 
 \bfsigma(u_i) \cdot\bfn_F\} [w_h]
 + \lim_{\delta \rightarrow 0}
  \sum_{\gamma \in \Gamma} \sum_{F\in\mathcal{F}_h^\gamma}
  \sum_{i\in I_\gamma}
  \int_F (\mathcal{K}_{\delta}^\mathfrak{d} 
 \bfsigma(u))|_{\Omega_i} \cdot \bfn_i \, 
 w_h|_{\Omega_i}\\
 & = \lim_{\delta \rightarrow 0}
 \left(
\sum_{i=1}^N 
\sum_{F \in \mathcal{F}_{h,i}^{\mathrm{int}}\cup \mathcal{F}_{h,i}^{\partial\mathcal{G}}} \int_{F } \{ \mathcal{K}_{\delta}^\mathfrak{d} 
 \bfsigma(u_i) \cdot\bfn_F\} [w_h]
 +\sum_{\gamma \in \Gamma} \sum_{F\in\mathcal{F}_h^\gamma}
  \sum_{i\in I_\gamma}
  \int_F (\mathcal{K}_{\delta}^\mathfrak{d} 
 \bfsigma(u))|_{\Omega_i} \cdot \bfn_i \, 
 w_h|_{\Omega_i}
 \right)\\
 & = \lim_{\delta \rightarrow 0}
 \sum_{i=1}^N
 \sum_{K\in\mathcal{T}_{h,i}}
 \int_{\partial K} \mathcal{K}_{\delta}^\mathfrak{d} 
 \bfsigma(u_i)|_K \cdot \bfn_K \, w_h|_K.
\end{align*}
Next,  we apply Green's theorem locally on each element $K$, and the convergence of the mollified operators \eqref{eq:convergence_mollifiers} to obtain 
\begin{align}
\sum_{i =1}^N \sum_{F \in \mathcal{F}_{h,i}^{\mathrm{int}} \cup \mathcal{F}_{h,i}^{\partial \mathcal{G}}} \tau_F(u,w_h)   &  +  \sum_{\gamma \in \Gamma} \sum_{F \in \mathcal{F}_h^\gamma }  \tau^\gamma_{F}(u, w_h)   \\ & =
\lim_{\delta \rightarrow 0 } \sum_{i=1}^N  \sum_{K \in \mathcal{T}_{h,i}} \left(\int_{ K} \nabla \cdot ( \mathcal{K}_{\delta}^\mathfrak{b} \bm \sigma (u)) \, w_h + \int_{K }  \mathcal{K}_{\delta}^\mathfrak{d}  \bm \sigma (u) \cdot \nabla w_h \right)\nonumber \\
& = - \sum_{i=1}^N \int_{\Omega_i} f_i \, w_h +  \sum_{i=1}^N  \sum_{K \in \mathcal{T}_{h,i}} \int_{K }  \bm \sigma (u) \cdot \nabla w_h. \nonumber 
\end{align}
Rearranging the above equality finishes the proof. 
\end{proof}
\begin{lemma}[Continuity of $\mathcal{A}$] \label{lemma:continuity_A}
Let $v$ be such that $v_i \in H^{1+s}(\Omega_i)$ and $\nabla \cdot \bm \sigma_i(v_i)\in L^2(\Omega_i)$ for $1\leq i\leq N$.  
For any  $w_h \in V_h^p$, the following bound holds
\begin{equation}
\mathcal{A}(v, w_h) \lesssim \Big(\sum_{i=1}^N (\|\nabla v\|_{L^2(\Omega_i)}^2 + h^{2s}  \vert\nabla v\vert_{H^{s}(\Omega_i)}^2)  + h^2 \sum_{K\in\mathcal{T}_h} \|\nabla  \cdot \bm \sigma(v)\|_{L^2(K)}^2\Big)^{1/2} \|w_h\|_{\DG}.
\end{equation}
\end{lemma}
\begin{proof}
We fix $v$ such that $v_i \in H^{1+s}(\Omega_i)$ and $\nabla \cdot \bm \sigma_i(v_i)\in L^2(\Omega_i)$ for $1\leq i\leq N$. For the first term in the definition of $\mathcal{A}$, we use Cauchy-Schwarz's inequality:
\[
\sum_{i=1}^N \sum_{K\in\mathcal{T}_{h,i}}
\int_K \bm\sigma(v) \cdot \nabla w_h
\leq \kappa_\ast \Vert w_h \Vert_{\mathrm{DG}}
\Big( \sum_{i=1}^N \Vert \nabla v \Vert_{L^2(\Omega_i)}^2\Big)^{1/2},
\]
where $\kappa_\ast = \max_{i}\Vert \kappa \Vert_{L^\infty(\Omega_i)}$.
From \cite[Lemma 3.2]{ern2021quasi} (with $q=2,  p = \rho= 4/(2-2s)  $), we have that for any polynomial function $\phi_h$ defined on $F$  and for any $K\in\mathcal{T}_h$ 
\begin{equation}
|\langle  (\bm \sigma(v)\vert_K \cdot \bfn_K) \vert_F, \phi_h  \rangle_F| \lesssim (h_K^s\|\bm \sigma(v)\|_{L^\rho(K)} + h_K \|\nabla \cdot \bm \sigma(v)\|_{L^2(K)}) h_F^{-1/2} \|\phi_h\|_{L^2(F)}. 
\end{equation}
We now recall the Sobolev embedding $H^s(K) \hookrightarrow L^\rho(K)$ with the correct scaling with respect to the element diameter $h_K$, see for e.g. \cite[Lemma 11.7]{Ern:booki}
\[
\forall z\in H^s(K), \quad
\Vert z \Vert_{L^\rho(K)}
\lesssim h_K^{-s} \Vert z \Vert_{L^2(K)}
+ \vert z \vert_{H^{s}(K)}.
\]
Since $\kappa$ belongs to $L^\infty(\Omega)$ and $\nabla v|_{\Omega_i}$ belongs to $H^{s}(\Omega_i)$ for all $1\leq i\leq N$, we have
for any $K\in\mathcal{T}_h$:
\begin{align}
\|\bm \sigma (v)\|_{L^\rho(K)} \lesssim h_K^{-s} \|\nabla v\|_{L^{2}(K)} + \vert \nabla v\vert_{H^{s}(K)}. 
\end{align}
 Combining the two bounds above, we obtain that 
\begin{multline}\label{eq:genres}
|\langle  (\bm \sigma(v)\vert_K \cdot \bm n_K) \vert_F, \phi_h  \rangle_F| \lesssim ( \|\nabla v \|_{L^2(K)} + h_K^s\vert \nabla v\vert_{H^{s}(K)} + h_K \|\nabla \cdot \bm \sigma(v)\|_{L^2(K)}) \\ \times  h_F^{-1/2} \|\phi_h\|_{L^2(F)}.
\end{multline}
The above bound is critical to estimate the second and third terms in $\mathcal{A}(v,w_h)$, see \eqref{eq:extended_form_A}.
Let us give some details for the third term; the second term is handled in a similar manner.
\[
\sum_{\gamma \in \Gamma}
\sum_{F\in\mathcal{F}_h^\gamma}
\tau^\gamma_F(v,w_h)
\leq \sum_{\gamma \in \Gamma}
\sum_{F\in\mathcal{F}_h^\gamma}
\frac{1}{\mathrm{card}(I_\gamma)}
\sum_{(i,j)\in D_\gamma}
\sum_{K \in \mathcal{T}_F \cap (\Omega_i \cup\Omega_j)} \vert
\langle (\bm \sigma(v)|_K \cdot \bfn_K)|_F, 
[w_h]_{(i,j)}\rangle_F\vert.
\]
Applying \eqref{eq:genres} with $\phi_h = [w_h]_{(i,j)}$, we have
\begin{multline*}
\sum_{\gamma \in \Gamma}
\sum_{F\in\mathcal{F}_h^\gamma}
\tau^\gamma_F(v,w_h)
\lesssim
\left( \sum_{\gamma \in \Gamma}
\sum_{F\in\mathcal{F}_h^\gamma}
\frac{\eta_\gamma}{h_F} \sum_{(i,j)\in D_\gamma}
\Vert [w_h]_{(i,j)}\Vert_{L^2(F)}^2
\right)^{1/2}\\
\times
\left(
\sum_{\gamma \in \Gamma}
\sum_{F\in\mathcal{F}_h^\gamma}
\sum_{(i,j)\in D_\gamma}
\sum_{K \in \mathcal{T}_F \cap (\Omega_i \cup\Omega_j)}
(\|\nabla v \|_{L^2(K)}^2 + h_K^{2s}\vert \nabla v\vert_{H^{s}(K)}^2 + h_K^2 \|\nabla \cdot \bm \sigma(v)\|_{L^2(K)}^2)
\right)^{1/2},
\end{multline*}
which easily gives the desired upper bound.
\end{proof}
\begin{theorem}[Error estimate]\label{thm:lr_estim}
Under the assumption of \Cref{lemma:weak_Kirchhoff_condition} and the assumption that $\kappa$ is piecewise constant,   the following bound holds
\begin{equation}\label{eq:lr_estim}
\|u - u_h\|_{\DG} \lesssim h^s \left( \sum_{i=1}^N |u|_{H^{1+s}(\Omega_i)}^2 + \sum_{i=1}^N \|f\|_{L^2(\Omega_i)}^2\right)^{1/2}. 
\end{equation}
\end{theorem}
\begin{proof}
Let $\mathcal{I}_h u \in V_h^1$ be a quasi interpolant of the exact solution $u$ defined locally on each $\Omega_i$, $\mathcal{I}_h u \vert_{\Omega_i} = \mathcal{I}_{h,i} u\vert_{\Omega_i}$ where $\mathcal{I}_{h,i}$ can be chosen as the Scott--Zhang interpolation operator defined over $\Omega_i$. 
Define $\chi_h = u_h - \mathcal{I}_h u $. 
We use the coercivity property of $a$: 
\begin{align}
\|\chi_h\|^2_{\DG} \lesssim a(\chi_h , \chi_h) = a(u_h, \chi_h) - a(\mathcal{I}_h u, \chi_h) = \sum_{i=1}^N \int_{\Omega_i} f_i \chi_h \vert_{\Omega_i} -  a(\mathcal{I}_h u, \chi_h).
\end{align}
Using \Cref{lemma:weak_consistency}, we obtain 
\begin{align}
    \|\chi_h\|^2_{\DG} \lesssim   \mathcal{A}(u, \chi_h) - a(\mathcal{I}_h u, \chi_h). 
\end{align}
We now rewrite the second term above using \Cref{lemma:froms_agree_discrete}, the fact that the jump of $I_h u$ vanishes on any edge $F \in \mathcal{F}_{h,i}^\mathrm{int}$ and the fact that $I_h u$ is zero on the hypergraph boundaries $\mathcal{F}_{h,i}^{\partial \mathcal{G}}$. In particular, it is easy to see that 
\begin{align}
a(\mathcal{I}_h u ,\chi_h) = \mathcal{A}(\mathcal{I}_h u, \chi_h) - \sum_{\gamma \in \Gamma} \sum_{F \in \mathcal{F}_h^\gamma} \int_F \{ \bm \sigma (\chi_h) \} \odot [\mathcal{I}_h u] + \sum_{\gamma \in \Gamma} \sum_{F \in \mathcal{F}_h^\gamma}  \frac{\eta_\gamma}{h_F} \int_F [\mathcal{I}_h u ] \odot [\chi_h].
\end{align}
Therefore, we obtain that 
\begin{align} \label{eq:err_low_0}
\|\chi_h\|^2_{\DG} & \lesssim  \mathcal{A}(u - \mathcal{I}_h u, \chi_h) +  \sum_{\gamma \in \Gamma} \sum_{F \in \mathcal{F}_h^\gamma} \int_F \{ \bm \sigma (\chi_h) \} \odot [\mathcal{I}_h u] - \sum_{\gamma \in \Gamma} \sum_{F \in \mathcal{F}_h^\gamma}  \frac{\eta_\gamma}{h_F} \int_F [\mathcal{I}_h u ] \odot [\chi_h]\\ 
& = T_1 + T_2 +T_3. \nonumber
\end{align}
We proceed by bounding the terms on the right hand side of \eqref{eq:err_low_0} separately. For the first term, we use \Cref{lemma:continuity_A} followed by approximation properties  of $\mathcal{I}_h$ \cite{scott1990finite,ciarlet2013analysis} and the fact that $\kappa$ is piecewise constant, thus implying that $\nabla \cdot \kappa (\nabla I_h u)$ vanishes on each element:
\begin{equation*}
     \mathcal{A}(u - \mathcal{I}_h u, \chi_h)  \leq \left( h^s (\sum_{i=1}^N |u|_{H^{1+s}(\Omega_i)}^2)^{1/2} + h \|f\|_{L^2(\mathcal{G})}\right)\|\chi_h\|_{\DG}. 
\end{equation*}

%
\Bk
Next we bound $T_2$ by using Cauchy-Schwarz's inequality and the fact that $u\in H_0^1(\mathcal{G})$:
\begin{align*}
T_2 =  \sum_{\gamma \in \Gamma} \sum_{F \in \mathcal{F}_h^\gamma} \frac{1}{\mathrm{card}(I_\gamma)} \int_F \sum_{(i,j)\in D_\gamma} \{ \bm \sigma (\chi_h) \}_{(i,j)} \, [\mathcal{I}_h u - u]_{(i,j)}\\
\lesssim \left(
\sum_{\gamma \in \Gamma} \sum_{F \in \mathcal{F}_h^\gamma} \sum_{(i,j)\in D_\gamma}
h_F \Vert  \{ \bm \sigma (\chi_h) \}_{(i,j)}
\Vert_{L^2(F)}^2
\right)^{1/2}
\left(
\sum_{\gamma \in \Gamma} \sum_{F \in \mathcal{F}_h^\gamma} \sum_{(i,j)\in D_\gamma}
\frac{1}{h_F} \Vert   [\mathcal{I}_h u - u]_{(i,j)}
\Vert_{L^2(F)}^2
\right)^{1/2}.
\end{align*}
Clearly by a discrete trace inequality, the first factor is bounded above by $\Vert \chi_h\Vert_{\DG}$. For the second factor, we expand the jump across an edge $F$ shared by two elements that belong to $\Omega_i$ and $\Omega_j$ and utilize the following trace inequality on the neighboring elements $K$:
\begin{align*}
\|\mathcal I_h u  - u \|_{L^2(F)} & \lesssim h_F^{-1/2} \|\mathcal I_h u  - u\|_{L^2(K)} + h_F^{s} |\mathcal I_h u  - u|_{H^{1/2+s}(K)} \\
& \lesssim  \nonumber h_F^{1/2 + s} |u|_{H^{1+s}(K)}. 
\end{align*}
We then have:
\[
T_2 \lesssim h^s (\sum_{i=1}^N |u|_{H^{1+s}(\Omega_i)}^2)^{1/2} \|\chi_h\|_{\DG}. 
\]
We follow a similar argument to obtain the same bound for $T_3$ and we skip the details for brevity.
\end{proof}

\section{Numerical examples}
\label{sec:numer}

In this section, we evaluate the error convergence of the method in \Cref{ex:1d} and \Cref{ex:2d} where the solution exhibits $H^2$ regularity, and in \Cref{ex:lr} where the solution has  $V^{1+s}$ regularity.
In \Cref{ex:fancy1d,ex:fancy2d} we
compare the DG scheme with discretization by continuous
Lagrange (CG) elements, where the continuity constraint at the bifurcation nodes is enforced
by construction of the finite element space and the Kirchoff law 
is enforced weakly. For further
details on CG finite element method on graphs, we refer to \cite{arioli2018finite}.
%

Our implementation relies on FEniCS and in particular its support of finite element
method on manifolds with different topological and geometric degree \cite{rognes2013automating}.
Since the jump and average operators currently available in FEniCS assume that each facet is shared by exactly two cells, the assembly of the bilinear form with the jump and average operators
\eqref{eq:jumpavg22d} is handled by our library FEniCS$_{\mathrm{ii}}$ \cite{kuchta2020assembly}. In following, we set $\eta_{F,i}=10 p$, $\eta_\gamma=10 p$ for examples with edge-networks while in case of plane-networks we take $\eta_{F,i}=20 p$, $\eta_\gamma=20 p$.

\subsection{Manufactured solution on edge-network}
\label{ex:1d}

We consider a test case inspired by \cite{masri2024discontinuous} where the Poisson
  problem is solved on a graph $\mathcal{G}$ embedded
  in $\mathbb{R}^2$ (see left panel in \Cref{fig:mesh1d}).  Specifically, we have eight boundary points with coordinates $(0, 0)$,
  $(-1.5, 3)$, $(-0.5, 3)$, $(0.5, 3)$, $(1.5, 3)$
  $(-1, 3)$, $(0.5, 2)$ and $(1.5, 2)$ and we have three bifurcation points
  $\gamma_{(1, 2, 3)}$, $\gamma_{(2, 4, 5, 8)}$ and $\gamma_{(3, 6, 7, 9, 10)}$ located at 
    $(0, 1)$, $(-1, 2)$ and $(1, 2)$, respectively.
  Setting $\kappa=1$ the exact solution is
  given as
  \begin{equation}\label{eq:network_hdg_mms}
    u(x,y) = \begin{cases}
      y + \cos (2\pi y)  & \!\!(x, y)\in \Omega_1,\\
      2 + \frac{1}{2}\sqrt{2}(y-1)  & \!\!(x, y) \in \Omega_{2\leq i \leq 3},\\
      1 + \frac{1}{2}\sqrt{2} + \frac{1}{8}\sqrt{1.25}(y-2)+ \cos (2\pi y) & \!\!(x, y)\in \Omega_{4\leq i \leq 7},\\
      2 + \frac{1}{2}\sqrt{2} & \!\!(x, y)\in \Omega_{8\leq i \leq 10}     
    \end{cases}.
  \end{equation}    

  \Cref{fig:mms1d} contains the plots of the errors in the DG norm and $L^2$ norm as a function of the mesh size for the three variants of the DG method. We note that with the IPDG and NIPG variants, we apply the over-penalization $1/h_F^2$ instead of $1/h_F$,
  see e.g. \cite[Ch 2.8.2]{riviere2008discontinuous} to recover optimal rates in the $L^2$ norm. Without over-penalization, the rates are optimal in the DG norm, and suboptimal in the $L^2$ norm. 
  For each variant, the DG solutions are obtained for polynomial degree equal to $1, 2$ and $3$.  We
  observe that the scheme 
  yields $p$-order convergence in the DG norm and order $p+1$ convergence in the $L^2$ norm.


  \begin{figure}
    \centering
    \includegraphics[width=0.9\textwidth]{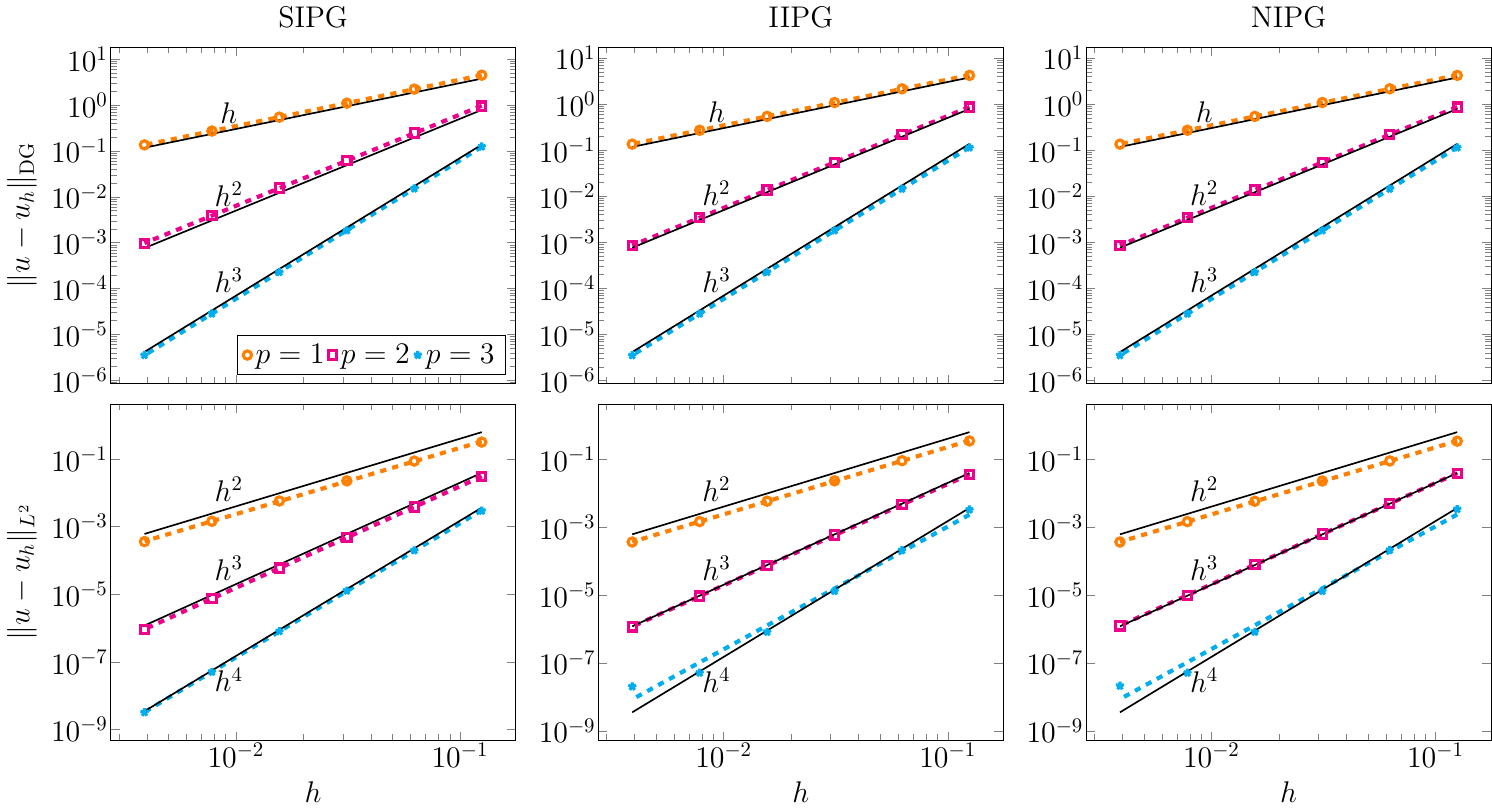}
    \caption{Poisson problem with exact solution \eqref{eq:network_hdg_mms} on the edge-network geometry in the left panel of \Cref{fig:mesh1d}.
    Top row displays the DG errors as a function of mesh size and bottom row displays the $L^2$ errors. 
      The SIPG, IIPG and NIPG variants are used with polynomial degrees $p = 1, 2, 3$.
    }
    \label{fig:mms1d}
  \end{figure}

\subsection{Manufactured solution on plane-network}
\label{ex:2d}
  We next apply the DG method to a network of planar surfaces. 
  To this end we extrude the geometry from \Cref{ex:1d} in the $z$-direction to obtain
  a collection of planar domains $\Omega_i$, $1\leq i \leq 10$, see the right panel of \Cref{fig:mesh1d}. The size of each domain in the $z$-direction is 1.
  We again consider the Poisson problem with $\kappa=1$ and the solution $u$, such that
  $u(x, y, z)=u(x, y) \sin(2\pi z)$ with $u(x, y)$ defined in \eqref{eq:network_hdg_mms}.
  We run similar experiments as in the previous example, and compute the errors in the DG norm and $L^2$ norm for the three variants of the DG method and three choices of polynomial degree. 
  With this setup \Cref{fig:mms2d} shows that our DG method converges optimally for $p=1, 2, 3$.  As in \Cref{ex:1d}, the IIPG and NIPG schemes used over-penalization.


  \begin{figure}  
    \includegraphics[width=0.9\textwidth]{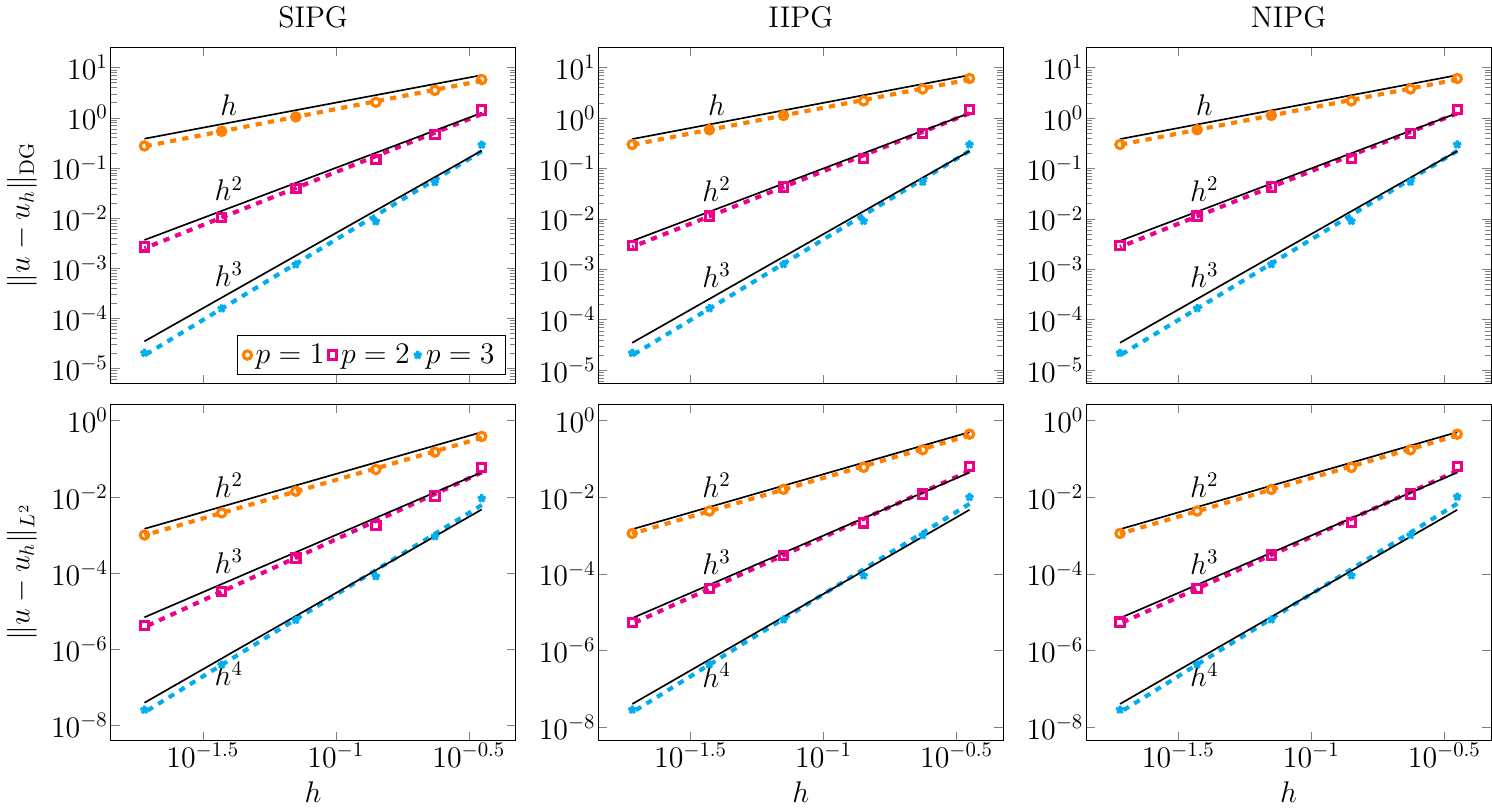}
    \caption{Poisson problem  on the plane-network geometry in the right panel of \Cref{fig:mesh1d}.
    Top row displays the DG errors as a function of mesh size and bottom row displays the $L^2$ errors. 
      The SIPG, IIPG and NIPG variants are used with polynomial degrees $p = 1, 2, 3$.
    }    
    \label{fig:mms2d}
    \end{figure}

\subsection{Manufactured solution with low regularity}
\label{ex:lr}
In this example, the plane-network $\mathcal{G}$ consists of an L-shaped domain $\Omega_1$ and
  a pair of rectangular domains (see \Cref{fig:mms_lr}).
  \[
  \begin{aligned}
    \Omega_1&=\left\{(x, y, z)\in\mathbb{R}^3, (x, y)\in(-1, 1)^2\setminus(-1, 0)^2, z=0\right\},\\
    \Omega_2&=\left\{(x, y, z)\in\mathbb{R}^3, (y, z)\in(0, 1)\times(-1, 1), x=-1\right\},\\
    \Omega_3&=\left\{(x, y, z)\in\mathbb{R}^3, (x, z)\in(0, 1)\times(-1, 1), y=-1\right\}.
  \end{aligned}
  \]
  We fix $0<s<1$ and define the exact solution $u$ by using
  the polar coordinates $(r,\theta)$ in the $xy$-plane:
  \begin{equation}\label{eq:network_low}
    u(r,\theta,z) = \begin{cases}
      r^{s}\sin(s\theta)  & \!\!\mbox{in}\, \Omega_1,\\
      r^{s}\sin(s\theta) + z  & \!\!\mbox{in}\, \Omega_2,\\
      r^{s}\sin(s\theta) + z^2  & \!\!\mbox{in}\, \Omega_3
    \end{cases}.
  \end{equation}      
  We note that $u\in V^{1+s}$. Let us also remark that at the bifurcation segments $u$ satisfies the continuity condition
  \eqref{eq:coupled_interfaceB} and a generalized Kirchoff's law, namely
  $\sum_{i\in I_\gamma} \bfsigma_i(u_i)|_\gamma\cdot \bfn_i = g_\gamma$
  with $g_\gamma\neq 0$.
  This non-zero source is accounted for by adding to the right-hand-side in \eqref{eq:scheme} the
  term term $\sum_{\gamma\in\Gamma} \sum_{i\in I_\gamma} \sum_{F\in\mathcal{F}_h^\gamma}\int_F g_\gamma \avg{v_h}_{(i, j)}$.

\Cref{fig:mms_lr} (left panel) displays the numerical solution with $p=1$ and the SIPG variant whereas the right panel shows the convergence of the scheme for
  different values of $s=0.25, 0.5, 0.75$. In agreement with \eqref{eq:lr_estim}
  in \Cref{thm:lr_estim} the errors in the DG-norm decay as $h^s$ in all three cases.
  \begin{figure}
    \centering
    \includegraphics[height=0.34\textwidth]{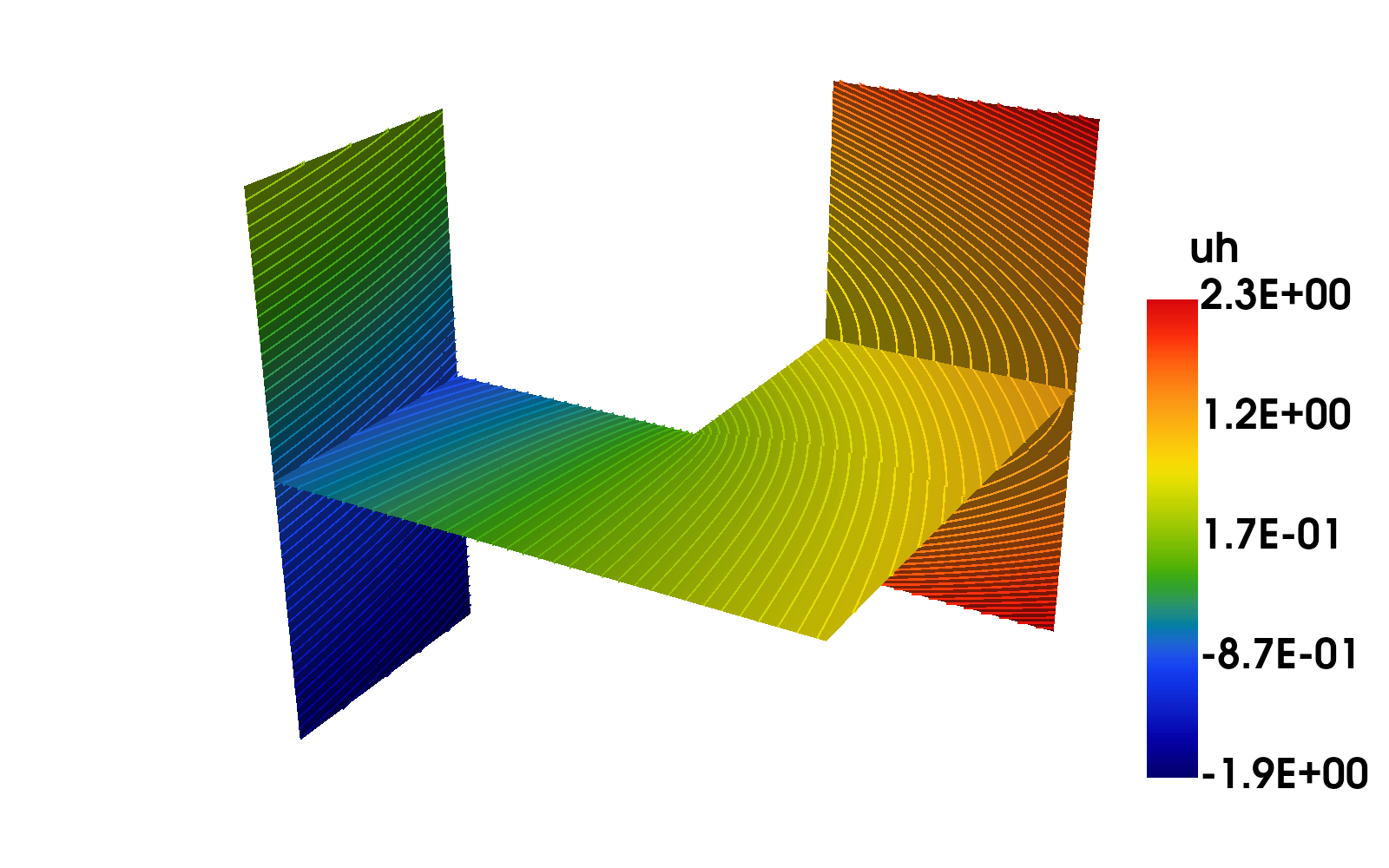}
    \includegraphics[height=0.34\textwidth]{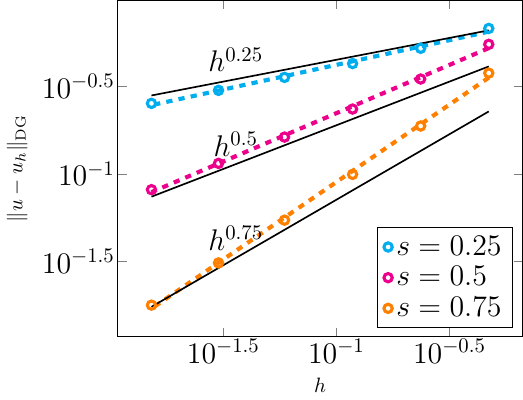}    
    \caption{
      Left: Numerical solution with the SIPG variant and $p=1$ for the  
      low regularity $V^{1+s}$ solution \eqref{eq:network_low}. Right: Errors in the DG norm as a function of mesh size for different values of $s$.}    
    \label{fig:mms_lr}
    \end{figure}

\subsection{Vascular edge-network}
\label{ex:fancy1d}

The edge-network $\mathcal{G}$ is part of the vascular network of a mouse cortex as
  given in the dataset~\cite{goirand2021network}. This network includes $57$  bifurcation nodes and there are at most four edges connected to a node. Given the geometry we solve the Poisson problem with $\kappa=1$ and
  with source term $f=1$ and homogeneous Dirichlet boundary condition.

  To verify convergence of the DG scheme, we compare the solution
  $u_h$ against a reference CG solution, $u^{*}$, which is obtained on the finest
  mesh. The SIPG variant is selected with $p=1$. Comparing the solutions in the $L^2$-norm
  we observe in \Cref{tab:fancy1d} that their difference converges with
  refinement. 
  
  We note that in this example the discrete solutions were obtained by preconditioned conjugate gradient solver 
  using as the preconditioner the AMG approximation of the Riesz map with respect to the inner product inducing the 
  norm \eqref{eq:dgnorm1d}. For all meshes, the solver converged in at most $15$ iterations. 
 
  
\begin{figure}[H]
    \begin{minipage}{0.45\textwidth}
    \centering
      \includegraphics[width=0.8\textwidth]{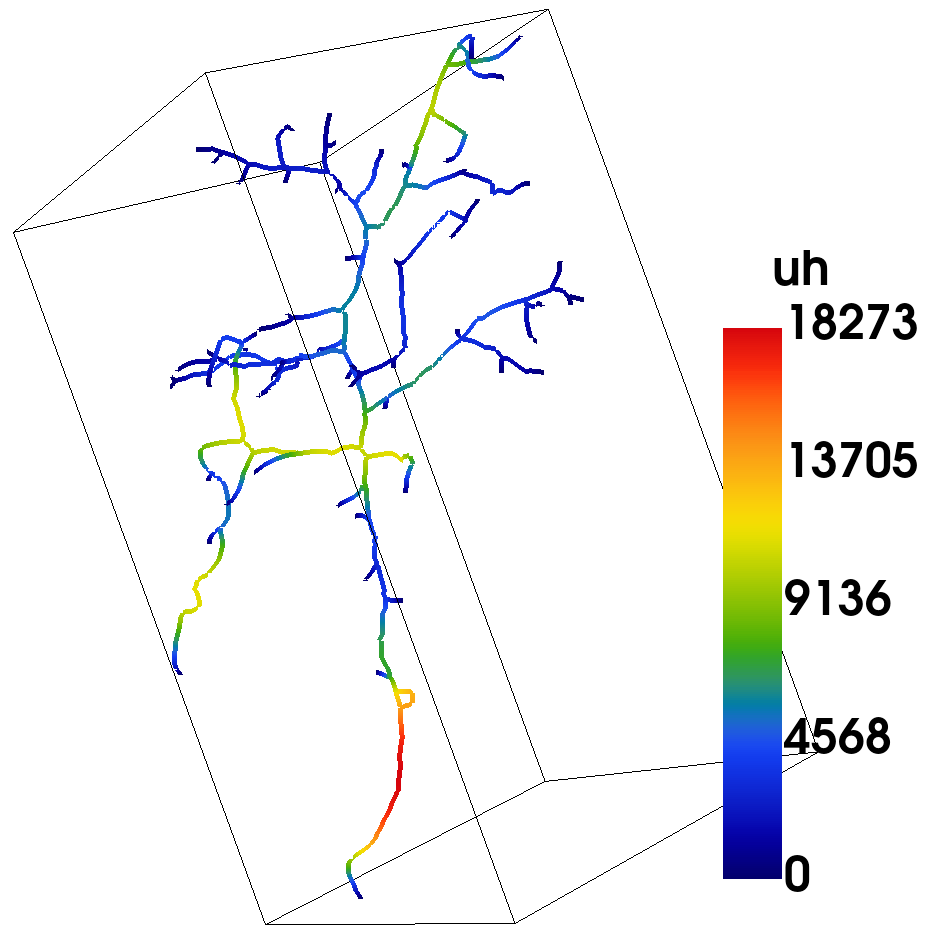}
    \end{minipage}
    \begin{minipage}{0.54\textwidth}
\hspace{0.5cm}      
\centering
\begin{tabular}{ccc}
\hline
$h$ &  $\mathrm{dim}(V_h^1)$  & $\frac{\lVert u^{*} - u_h\rVert_{L^2(\mathcal{G})}}{\lVert u^{*}\rVert_{L^2(\mathcal{G})}}$ \\
\hline
 4.30E-01 &  22392  & 5.51E-06 \\
 2.15E-01 &  44784  & 1.37E-06 \\
 1.08E-01 &  89568  & 3.28E-07 \\
 5.38E-02 & 179136  & 6.84E-08 \\
 2.69E-02 & 358272  & 2.81E-10 \\
 \hline
\end{tabular}
    \end{minipage}
    \caption{
      Poisson problem on vascular network from \cite{goirand2021network} solved by
      the DG scheme. Left panel: problem geometry and numerical solution. Right panel: relative $L^2$ errors with
      respect to the reference CG solution $u^{*}$.
    }
    \label{tab:fancy1d}
\end{figure}

\subsection{Complex plane-network}\label{ex:fancy2d}

  We obtain the plane-network $\mathcal{G}$ by subdividing a unit cube in $27$ cubes of size $1/3$,
  taking union of their respective boundaries and removing the parts that intersect with the boundary of the unit cube, see \Cref{tab:fancy2d}. The domain includes $36$ bifurcation segments, each connecting
  four planar surfaces. 

As in the previous example, we solve the Poisson problem with $\kappa=1$ and with source term $f=1$ and homogeneous Dirichlet boundary condition and we use the CG solution on the finest mesh as reference solution.  The SIPG variant is chosen with $p=1$. The left panel of \Cref{tab:fancy2d} displays the numerical solution and the right panel shows the convergence of the method as the mesh size decreases. 
Using AMG approximation of the Riesz map due to \eqref{eq:dgnorm2d} as the preconditioner for the conjugate gradient method the solutions 
were obtained after $24$ iterations on the coarsest mesh and $25$ iterations on all the finer meshes.

  
\begin{figure}[H]
    \begin{minipage}{0.40\textwidth}
    \centering
      \includegraphics[width=1.0\textwidth]{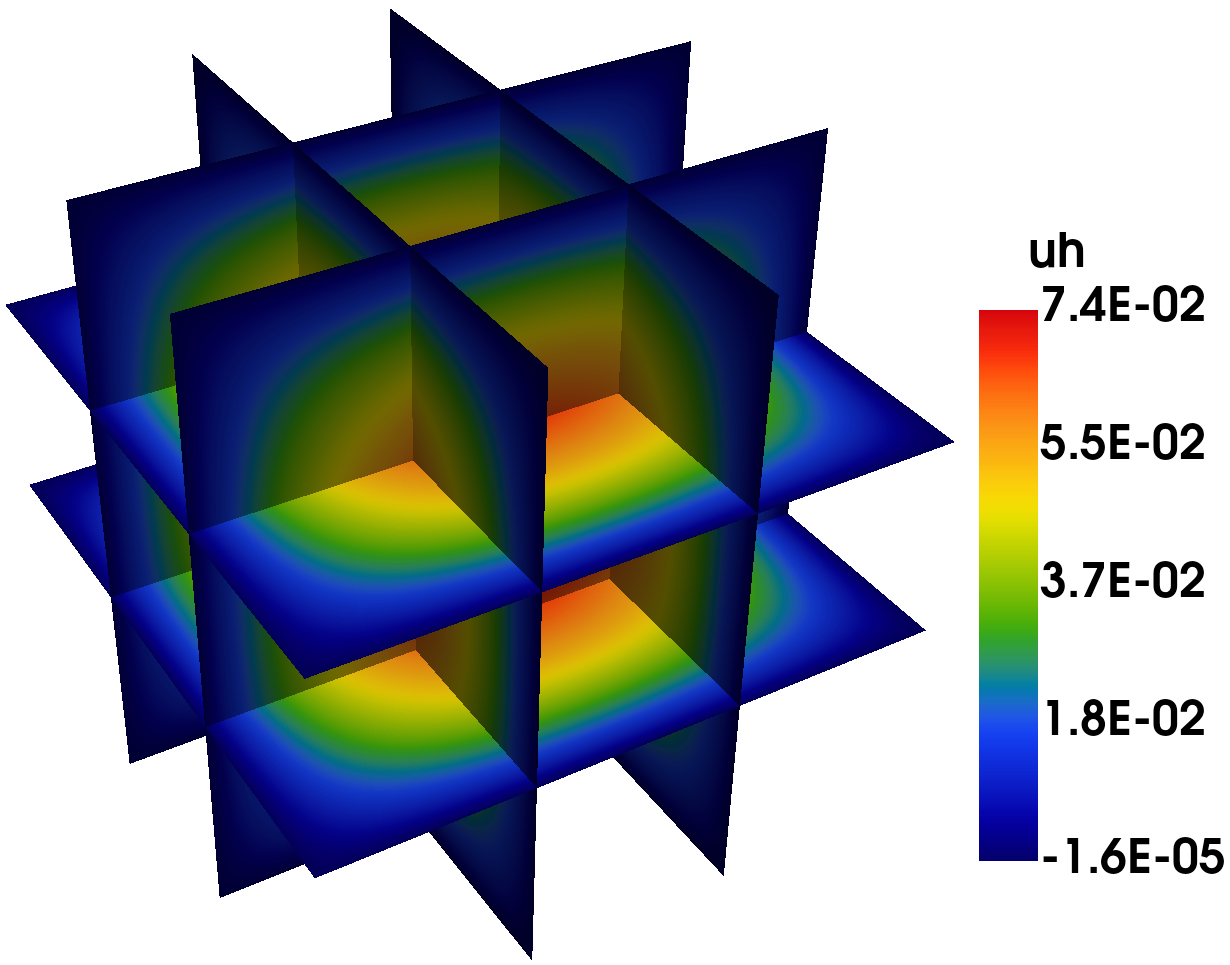}
    \end{minipage}
    \begin{minipage}{0.58\textwidth}
    \centering
\hspace{0.5cm}
\begin{tabular}{ccc}
\hline
$h$ &  $\mathrm{dim}(V_h^1)$ &   $\frac{\lVert u^{*} - u_h\rVert_{L^2}}{\lVert u^{*}\rVert_{L^2}}$ \\
\hline
 2.36E-01 & 1296  &   6.46E-02  \\
 1.18E-01 & 5184  &  1.69E-02  \\
 5.89E-02 & 20736 &  4.21E-03  \\
 2.95E-02 & 82944 &  9.82E-04  \\
 1.52E-02 & 311364&  5.82E-05  \\
\hline
\end{tabular}
    \end{minipage}
    \caption{
      Poisson problem on a plane-network solved by
      the DG scheme with $p=1$.  Left panel: problem geometry and numerical solution. Right panel: relative $L^2$ error with
      respect to the reference CG solution $u^{*}$.
    }
    \label{tab:fancy2d}
  \end{figure}  


\section{Conclusions}

This work addresses the numerical analysis of discontinuous polynomial approximations of elliptic problems on hypergraphs made of either segments or of planar surfaces. Using a manipulation of the Kirchoff conditions and a new definition of averages and jumps, the terms handling the bifurcation nodes are seamlessly written in the DG form. A priori error bounds are derived; they require technical functional analysis tools for the low regularity case. Convergence is verified for several numerical examples. This work is motivated by the mixed dimensional problem of PDEs valid in three-dimensional domains that are coupled to PDEs valid on hypergraphs; the numerical analysis of these problems is future work.

\bibliographystyle{amsplain}
\bibliography{references}

\end{document}